\documentclass[a4paper,10pt]{article}
\usepackage[vlined]{algorithm2e}
\SetFuncSty{textsc}
        \SetKwFunction{frun}{Run}
        \SetKwHangingKw{arun}{\frun}
        \SetKwFunction{fset}{Set}
        \SetKwHangingKw{aset}{\fset}
        \SetKwFunction{fselect}{Select}
        \SetKwHangingKw{aselect}{\fselect}
        \SetKwFunction{fcompute}{Compute}
        \SetKwHangingKw{acompute}{\fcompute}
        \SetKwFunction{fsolve}{Solve}
        \SetKwHangingKw{asolve}{\fsolve}
        \SetKwFunction{festimate}{Estimate}
        \SetKwHangingKw{aestimate}{\festimate}
        \SetKwFunction{fmark}{Mark}
        \SetKwHangingKw{amark}{\fmark}
        \SetKwFunction{frefine}{Refine}
        \SetKwHangingKw{arefine}{\frefine}
        \SetKwFunction{compute}{Compute}
        \SetKwFunction{set}{Set}

        \newenvironment{algof}%
{\algorithm}%
   {\endalgorithm}

\title{Adaptive Morley FEM for the von K\'{a}rm\'{a}n  \\
equations with optimal convergence rates}
\author{Carsten Carstensen\footnote{Department of Mathematics, 
Humboldt-Universit\"{a}t zu Berlin, 10099 Berlin, Germany.
Distinguished Visiting Professor, Department of Mathematics, Indian institute of 
Technology Bombay, Powai, Mumbai-400076. Email. cc@math.hu-berlin.de}        
\quad\text{and}\quad 
Neela Nataraj\footnote{Department of Mathematics, Indian Institute of Technology Bombay, Powai, Mumbai 400076, India. Email. neela@math.iitb.ac.in}
}
\usepackage{amsmath,amsthm,amssymb,enumerate}
\usepackage[T1]{fontenc}
\usepackage{gensymb}
\usepackage[square,sort&compress,comma,numbers]{natbib}
\usepackage[sc]{mathpazo}
\usepackage{multirow} 
\usepackage{newtxtext,newtxmath}
\usepackage{subfig}
\usepackage{graphicx}
\usepackage{epstopdf}
\usepackage{cancel}
\usepackage[margin=3cm]{geometry}
\usepackage{tikz}
\usepackage{caption}
\usetikzlibrary{shapes,calc}
\usepackage{verbatim}
\usepackage{mathrsfs}
\usepackage{accents}
\usepackage[utf8]{inputenc}

\usetikzlibrary{decorations.pathmorphing}
\usetikzlibrary{decorations.pathreplacing}
\usetikzlibrary{positioning}
\usetikzlibrary{shapes}
\usetikzlibrary{arrows}
\usetikzlibrary{patterns}
\usetikzlibrary{fadings}
\usetikzlibrary{plotmarks}
\usetikzlibrary{calc}
\usetikzlibrary{intersections}
\tikzstyle{every picture}+=[font=\footnotesize]
\usepackage{paralist}
\usepackage{bbm}
\usepackage{latexsym}           
\usepackage{enumerate}
\usepackage{enumitem}

\setlist{noitemsep, topsep=0.8ex, partopsep=0pt
	, leftmargin=3em}
\setlist[1]{labelindent=\parindent}

\newlist{axioms}{enumerate}{1}
\setlist[axioms]{font=\bfseries}

\newlist{alphenum}{enumerate}{1}
\setlist[alphenum]{label=\textbf{(\alph*)}, leftmargin=4em}

\newlist{alphienum}{enumerate}{1}
\setlist[alphienum]{label=\textit{(\alph*)}}

\newlist{romanenum}{enumerate}{1}
\setlist[romanenum]{label=\textit{(\roman*)}}

\newlist{romaninenum}{enumerate*}{1}
\setlist[romaninenum]{label=\textit{(\roman*)}}
\usepackage[vlined]{algorithm2e}
\SetKwIF{If}{ElseIf}{Else}{if}{}{else if}{else}{endif}
\SetKwFor{For}{for}{}{endfor}
\usepackage[noabbrev, capitalise]{cleveref}
\usepackage{tabu}
\crefname{equation}{\unskip}{\unskip}
\creflabelformat{equation}{#2(#1)#3}
        \SetFuncSty{textsc}
        \SetKwFunction{frun}{Run}
        \SetKwHangingKw{arun}{\frun}
        \SetKwFunction{fset}{Set}
        \SetKwHangingKw{aset}{\fset}
        \SetKwFunction{fselect}{Select}
        \SetKwHangingKw{aselect}{\fselect}
        \SetKwFunction{fcompute}{Compute}
        \SetKwHangingKw{acompute}{\fcompute}
        \SetKwFunction{fsolve}{Solve}
        \SetKwHangingKw{asolve}{\fsolve}
        \SetKwFunction{festimate}{Estimate}
        \SetKwHangingKw{aestimate}{\festimate}
        \SetKwFunction{fmark}{Mark}
        \SetKwHangingKw{amark}{\fmark}
        \SetKwFunction{frefine}{Refine}
        \SetKwHangingKw{arefine}{\frefine}
        \SetKwFunction{compute}{Compute}
        \SetKwFunction{set}{Set}
\newtheorem{thm}{Theorem}[section]
\newtheorem{cor}[thm]{Corollary}
\newtheorem{lem}[thm]{Lemma}

\theoremstyle{definition}

\theoremstyle{remark}
\newtheorem{rem}{Remark}[section]

\numberwithin{equation}{section}

\newcommand{\bN}{\mathbb N}
\newcommand{\bR}{\mathbb R}

\newcommand{\bT}{\mathbb T}
\newcommand{\bV}{\mathcal V}

\newcommand{\mM}{\mathcal M}

\newcommand{\cE}{\mathcal E}

\newcommand{\cV}{\mathcal V}
\newcommand{\cT}{\mathcal T}

\newcommand{\cN}{\mathcal N}

\newcommand{\cM}{\mathcal M}

\newcommand{\pw}{\rm {pw}}

\newcommand{\Lstab}[0]{\ensuremath{\Lambda_{\mathrm{1}}}}
\newcommand{\Lred}[0]{\ensuremath{\Lambda_{\mathrm{2}}}}
\newcommand{\Lqo}[0]{\ensuremath{\Lambda_{\mathrm{4}}}}

\newcommand{\LdRel}[0]{\ensuremath{\Lambda_{\mathrm{3}}}}

\newcommand{\abs}[1]{\left\lvert #1 \right\rvert}

\newcommand{\bH}{\boldsymbol{H}}

\newcommand{\hto}{H^2_0(\Omega)}

\newcommand{\hT}{\widehat{\mathcal{T}}}

\newcommand{\NC}{\text{pw}}

\newcommand{\vket}{von K\'{a}rm\'{a}n equations }

\newcommand{\integ}{\int_\Omega}
\newcommand{\sit}{\sum_{K \in\mathcal{T}}\int_K}

\newcommand{\ma}{{\rm max}}
\newcommand{\fl}{\;\text{ for all }}

\newcommand{\half}{\frac{1}{2}}
\newcommand{\trinl}{\ensuremath{|\!|\!|}}
\newcommand{\trinr}{\ensuremath{|\!|\!|}}

\newcommand{\dx}{{\rm\,dx}}

\newcommand{\ds}{{\rm\,ds}}

\newcommand{\cof}{{\rm cof}}

\newcommand{\wv}{{\widehat{v}}}

\newcommand{\TT}[1]{\mathbb T\left(#1\right)}

\newcommand{\TO}{\mathbb T}

\DeclareMathOperator{\E}{\mathcal{E}}
\newcommand{\M}{\text{M}}

\newcommand{\Tl}[1]{\mathcal T_{#1}}

\newcommand{\Ml}[1]{\mathcal M_{#1}}

\newcommand{\T}{\mathcal{T}}

\newcommand{\NO}[0]{\mathbb{N}_0} 

\begin{document}	
\maketitle
\begin{abstract}
The adaptive nonconforming Morley finite element method (FEM) approximates a
regular solution to the von K\'{a}rm\'{a}n equations with optimal convergence rates
for sufficiently fine triangulations and small bulk parameter in the D\"orfler marking.
This follows from the general axiomatic framework with the key arguments  of stability, reduction, discrete reliability, and quasiorthogonality of an explicit residual-based error estimator. Particular attention is on the nonlinearity and the piecewise Sobolev embeddings 
required in the resulting trilinear form in the weak formulation of the nonconforming
discretisation.  The discrete reliability follows with a 
conforming companion for the discrete Morley functions from the medius analysis. The quasiorthogonality also
relies on a novel piecewise $H^1$ a~priori error estimate and a careful analysis of the nonlinearity.  
\end{abstract}



\section{Introduction}

This paper establishes the optimal convergence rates of an adaptive finite element method for the nonconforming Morley approximation to a regular solution of the semilinear von K\'{a}rm\'{a}n equations in a bounded polygonal 
Lipschitz domain $\Omega$ in the plane. 
\noindent {\subsection {The mathematical model} The mathematical model describes the deflection $u$
of very thin elastic plates by a semi-linear system of fourth-order partial differential equations:  
For a given load function $f\in L^2(\Omega)$, seek $u,v \in H^2_0(\Omega)$ such that
\begin{equation}\label{vkedG}
\Delta^2 u =[u,v]+f \text{ and }\Delta^2 v =-\half[u,u] \quad
 \text{ in } \Omega.
\end{equation}
Here and throughout the paper, $\Delta^2$  denotes the biharmonic operator with  
$\displaystyle\Delta^2\varphi =\varphi_{xxxx}+2\varphi_{xxyy}+\varphi_{yyyy}$   and  
$[\bullet,\bullet]$ denotes the von K\'{a}rm\'{a}n bracket with 
$[\eta,\chi] =\eta_{xx}\chi_{yy}+\eta_{yy}\chi_{xx}-2\eta_{xy}\chi_{xy}=\cof(D^2\eta):D^2\chi$
 for the co-factor matrix $\cof(D^2\eta)$ of $D^2\eta$ (the colon $:$ denotes the scalar product between $2\times 2$ matrices) for smooth functions $\varphi, \eta, \chi $ 
 and their partial derivatives $\varphi_{xx} $ etc. 

\subsection{A brief literature review} The existence of solutions, regularity, and bifurcation phenomena are discussed in \cite{CiarletPlates,Knightly,BergerFife} and the references therein. 
The weak solutions $u,v\in \hto$ to the \vket \eqref{vkedG} belong to 
$\hto\cap H^{2+\gamma}(\Omega)$ with the index of elliptic regularity $\gamma>1/2$ determined by the interior angles of the polygonal boundary $\partial\Omega$ with 
$\gamma=1$ if $\Omega$ is convex \cite{BlumRannacher}. In view of the first-order convergence 
rates even if $\gamma>1$ owing to the quadratic ansatz functions in this fourth-order problem, suppose that $\gamma$ is the minimum of one and the regularity index of the polygonal domain $\Omega$, so $1/2<\gamma\le 1$.

The major challenges in the numerical analysis of \eqref{vkedG} are the non-linearity 
and the higher-order nature of the equations.  
The papers \cite{Brezzi,Miyoshi,Quarteroni,Reinhart} study the approximation and error bounds for {\it regular} solutions to \eqref{vkedG} for conforming, mixed, and hybrid FEMs. 
Meanwhile,   nonconforming FEMs \cite{GMNN_Morley}, a $C^0$ interior penalty method \cite{BS_C0IP_VKE}, and discontinuous Galerkin FEMs \cite{CCGMNN18} have been 
investigated and  \cite{CCGMNN_Semilinear} suggests an abstract framework for 
  {\it a priori}  and {\it a posteriori} error control applicable to  the von K\'{a}rm\'{a}n equations. The most popular of those is the nonconforming Morley FEM for it is parameter free and is as simple as quadratic Lagrange finite elements; the reader may consider the finite element program in \cite[Sec. 6.5]{CCDGJH14} with less than 30 lines of Matlab.
 
Given the reduced elliptic regularity on nonconvex polygons with $\gamma<1$, 
the convergence for a quasiuniform triangulation with (maximal) mesh-size 
$h_{\max}$ is not better than
 $h_{\max}^\gamma$  and adaptive mesh-refining is mandatory. 
Little is known in the literature  about adaptive finite element methods (FEMs) 
and their convergence rates for semilinear problems.   For particular  strictly monotone and Lipschitz continuous 
operators, the residuals are similar to their  linear relatives  
with  unique exact (resp. discrete solutions) and optimal convergence rates 
are known \cite{Pretetal,cc14,ccdgms}. Besides the p-Laplacian \cite{BDK2011}, 
there are merely convergence proofs (but no optimal rates) for semilinear second-order
problems \cite{HPZ}. 
The other known results are for eigenvalue problems for the Laplacian or the bi-Laplacian
with a perturbation of the right-hand side $f:=\lambda u$ 
in the exact problem replaced by $f_h:=\lambda_h u_h$ on the discrete level. 
Since $f-f_h:=\lambda u-\lambda_h u_h$ is a higher-order perturbation,   the axioms of adaptivity \cite{cc14,cc_hella_18}  lead to  optimal convergence rates for sufficiently small
mesh-sizes.

\subsection{Contributions} This paper provides the first rate-optimal adaptive algorithm  
 for the    \vket  \eqref{vkedG}.   In the absence of further structural 
information (e.g.,  the Rayleigh-Ritz principle for symmetric eigenvalue problems  in 
\cite{CarstensenGedicke2014,CCDG14} or small loads as in \cite{Knightly}) 
the mesh-size has to be sufficiently small to  guarantee the existence of a unique discrete
solution   $\Psi_\M\in \bV(\cT)$ close to   $\Psi\in \bV$. This is achieved in a pre-asymptotic step of
the proposed adaptive algorithm AMFEM
by uniform mesh refinements until the mesh-size  is smaller than or equal to some input parameter 
$\delta$. The standard adaptive FEM with solve, estimate, mark, and refine applies thereafter 
with a bulk parameter $\theta$ in the D\"orfler marking  
\cite{BinevDahmenDevore04,cc14,cc_hella_18,CKNS08,Stev07}.  
This paper presents  an  adaptive algorithm AMFEM and establishes optimal convergence rates 
for all sufficiently small positive $\delta$ and $\theta$.

For a triangle $T$ of area $|T|$,
the  Morley finite element approximation  $\Psi_\M:=(u_\M,v_\M)$  to \eqref{vkedG}
leads to a volume residual 
$\mu(T) :=|T| (\|  f+ [u_\M,v_\M] \|^2_{L^2(T)} +\|   [u_\M,u_\M] \|^2_{L^2(T)} )^{1/2}$. 
 Despite the mesh-size factor $h_T^2=|T|$,  this volume contribution 
is {\em not} of higher order: A (standard, whence undisplayed) local efficiency error analysis reveals that $\mu(T) \lesssim \|D^2(\Psi-\Psi_\M)\|_{L^2(T)} $ plus data oscillations unlike in the aforementioned eigenvalue error analysis \cite{DG_Morley_Eigen}. The remedy in this paper consistes of a new reduction property of the sum of the  $\mu(T)$ contributions 
up to small perturbations and a novel {\it a priori} error estimate in the piecewise $H^1$ norm of $\Psi-\Psi_\M$.

This paper establishes the adaptive Morley FEM as the first scheme with guaranteed optimal rates for the von K\'{a}rm\'{a}n plate model. It is therefore suggested as the method of choice for a nonconvex domain.

%

\subsection{Structure of the paper} The outline of the remaining parts of this paper reads as follows. Section 2 recalls  some known preliminaries 
about the analysis of the von K\'{a}rm\'{a}n equations and  tools from the medius analysis of the Morley FEM. 
Section 3 presents the  {\it a priori} error estimates with a novel 
piecewise $H^1$ norm error estimate. Section~4 recalls the  
explicit residual-based error estimator from \cite{CCGMNN_Semilinear}  and introduces 
the adaptive algorithm AMFEM for the nonconforming Morley FEM. 
Section~5 gives details and proofs of  stability, reduction, discrete reliability,  and quasiorthogonality. This 
and  
\cite{cc14,cc_hella_18} 
guarantee the optimal convergence rates of the proposed adaptive Morley FEM. 
The outline is restricted  to two space dimensions {since} the von  K\'{a}rm\'{a}n equations are intrinsically two-dimensional.

\subsection{General Notation} Standard notation of  Lebesgue and Sobolev spaces, 
their norms, and $L^2$ scalar products  applies throughout the paper
such as the abbreviations $\|\bullet\|_p$ for 
$\|\bullet\|_{L^p(\Omega)}$  { (resp. $\|\bullet\|_{p,\omega}$ for
$\|\bullet\|_{L^p(\omega)}$ where $\omega\subset\Omega$}) and $\|\bullet\|_{m,p}$  for 
$\|\bullet\|_{W^{m,p}(\Omega)}$
and the local (resp. piecewise) 
version  $\|\bullet\|_{m,p,\omega}:=\|\bullet\|_{W^{m,p}(\omega)}$
(resp. $\|\bullet\|_{m,p,\omega,\text{pw}}$)
 and for the related seminorms.
The notation ${\bf H}^s(\Omega)$ (resp. ${\bf L}^p(\Omega)$)  denotes the product space  
 $H^{s}(\Omega) \times H^s(\Omega)$ (resp. $L^p(\Omega) \times L^p(\Omega)$) for 
 $s \in {\mathbb R}$ (resp. $1\le p \le\infty$). 
 The triple norm  $\trinl \bullet \trinr:=|\bullet|_{H^{2}(\Omega)}$ is the energy norm
 and  $\trinl \bullet \trinr_{\text{pw}}:=|\bullet|_{H^{2}(\cT)}:=\| D^2_\text{pw}\bullet\|_2$ 
 is its piecewise version with the piecewise Hessian $D_\text{pw}^2$. 
 {A weighted Young inequality for  positive  $ a,b,\epsilon>0$ means the arithmetic-geometric mean inequality 
\begin{equation}\label{eq:weighted}
 ab \le \frac{a^2}{2 \epsilon} + \frac{ \epsilon b^2}{2}.
 \end{equation}} 
The notation $A \lesssim B$ abbreviates $A \leq CB$ for some positive generic constant $C$, 
which depends only on the initial triangulation $\cT_{\text{init}}$ and on the regular solution $\Psi$;
 $A\approx B$ abbreviates $A\lesssim B \lesssim A$.

\section{Morley FEM for the von K\'{a}rm\'{a}n equations}
The first  subsection is devoted to the mathematical model  of the von K\'{a}rm\'{a}n equations and is followed by a subsection on triangulations and  discrete spaces.  The third subsection recalls 
interpolation and enhancement for  Morley functions  and the fourth collects further preliminaries. 
 
\subsection{The von K\'{a}rm\'{a}n equations}
Given $f\in L^2(\Omega)$ in a bounded polygonal Lipschitz domain $\Omega\subset\bR^2$, the weak formulation of the von K\'{a}rm\'{a}n equations \eqref{vkedG} seeks  $u,v\in  \hto$ such that
\begin{subequations}\label{wform}
	\begin{align}
	& a(u,\varphi_1)+ b(u,v,\varphi_1)+b(v,u,\varphi_1)= (f,\varphi_1)_{L^2(\Omega)}   \fl\varphi_1\in \hto, \label{wforma}\\
	& a(v,\varphi_2)-b(u,u,\varphi_2)   =0            \fl\varphi_2 \in \hto.\label{wformb}
	\end{align}
\end{subequations}
Here and throughout the paper, for all $\eta,\chi,\varphi\in \hto$ and all   $\Xi=(\xi_1,\xi_2),\Theta=(\theta_1,\theta_2)$, 
$\Phi=(\varphi_1,\varphi_2)\in  \bV:= \hto \times \hto$  (endowed with the product
norm also denoted by $\trinl\bullet\trinr$) 
set
\begin{align*}
&a(\eta,\chi):=\integ D^2 \eta:D^2\chi\dx  \; \; \text{and}\; \; b(\eta,\chi,\varphi):=-\half\integ [\eta,\chi]\varphi\dx, \\
& A(\Theta,\Phi):=a(\theta_1,\varphi_1)+a(\theta_2,\varphi_2), \\
&  B(\Xi,\Theta,\Phi):=b(\xi_1,\theta_2,\varphi_1)+b(\xi_2,\theta_1,\varphi_1)-b(\xi_1,\theta_1,\varphi_2), \\ 
& {\rm and}
\; F(\Phi):= (f, \varphi_1)_{L^2(\Omega)}.
\end{align*}
The boundedness and ellipticity properties \cite{Brezzi, GMNN_BFS} read, 
for all $\Theta,\Phi \in \bV$,
\begin{align*}
&{A}(\Theta,\Phi)\leq \trinl\Theta\trinr \: \trinl\Phi\trinr,\;  {A}(\Theta,\Theta) = \trinl\Theta\trinr^2, \; B(\Xi, \Theta, \Phi) \lesssim  \trinl\Xi\trinr \: \trinl\Theta\trinr \: \trinl\Phi\trinr .
\end{align*}
The trilinear form  $b(\bullet,\bullet,\bullet)$ is symmetric in first two 
variables and so is $B(\bullet,\bullet,\bullet)$. 
The vector form of \eqref{wform} seeks $\Psi=(u,v)\in \bV$ with $N(\Psi)=0$ for 
the nonlinear function $N: \bV \rightarrow \bV^*$, 
\begin{equation}\label{vform_cts}
N(\Psi;\Phi):=A(\Psi,\Phi)+B(\Psi,\Psi,\Phi)-F(\Phi)=0\fl  \Phi\in\bV.
\end{equation}

\begin{thm}[Regularity  \cite{BlumRannacher, ngr}]\label{ap}
Given any $f\in H^{-1}(\Omega)$ with norm $\| f \|_{-1}=\| f \|_{H^{-1}(\Omega)}$, there exists at least one
solution $\Psi$ to  \eqref{vform_cts} and any such $\Psi$ belongs to 
$ {\bf H}^{2 + \gamma} (\Omega)\cap\bV $ for some elliptic 
regularity index $\gamma \in (1/2,1]$ 
with  
$\trinl\Psi\trinr \lesssim  \| f \|_{{-1}} $ and
$ \|\Psi\|_{H^{2 + \gamma}(\Omega)}   \lesssim \| f \|^3_{-1} +
 \| f \|_{-1}$. \hfill \qed
\end{thm} 

A solution $\Psi$ to  \eqref{vform_cts} is called  a {\em regular solution} if the Frech\'et derivative 
$DN(\Psi)\in L(\bV;\bV^*)$ of $N$ at 
$\Psi$ is an isomorphism. It is also known \cite{Knightly}
that for sufficiently small $f$, the solution is unique  and 
is a regular solution; but this paper aims at a local approximation of an
arbitrary regular solution. 

The  Fr\'echet derivative $DN(\Psi)= A(\bullet,\bullet)+2 B(\Psi,\bullet,\bullet)$ 
of the operator $N$ at the regular solution $\Psi$  is an isomorphism and   this is equivalent to an 
$\inf$-$\sup$ condition 
\begin{align}\label{inf-sup}
0<\beta:=\inf_{\substack{\Theta\in \bV\\ \trinl\Theta\trinr=1}}\sup_{\substack{\Phi\in \bV\\ \trinl\Phi\trinr=1}}\big{(}A(\Theta,\Phi)+2B(\Psi,\Theta,\Phi)\big{)}.
\end{align}

Given a regular solution $\Psi \in {\bV}$ to \eqref{wform}
and any  ${\boldsymbol g} \in {\bf H}^{-1}(\Omega):={\bf H}^{1}_0(\Omega)^*$
(with duality brackets 
$\langle \bullet, \bullet \rangle_{{H}^{-1}(\Omega) \times {H}^1_0(\Omega)}$)   
 the dual linearised problem 
seeks  ${\boldsymbol \zeta} \in \bV$ with 
\begin{align}
{\mathcal A}(\Phi,  {\boldsymbol \zeta}):= 
A (\Phi, {\boldsymbol \zeta} ) + 
2 B (\Psi, \Phi, {\boldsymbol \zeta})=\langle {\boldsymbol  g}, 
\Phi \rangle_{{ H}^{-1}(\Omega) \times {H}^1_0(\Omega)} \; \; \text{for all } \Phi \in 
\bV \label{duallinearized}.
\end{align}
This problem is  well-posed \cite{Brezzi,GMNN_BFS} and satisfies,
for the same elliptic regularity index $\gamma \in (1/2,1]$ as in Theorem~\ref{ap}, that
\begin{align} \label{regularity}
\trinl {\boldsymbol \zeta}  \trinr \lesssim \| {\boldsymbol \zeta}  \|_{H^{2+\gamma}(\Omega)} \lesssim 
\| {\boldsymbol g}\|_{-1}.
\end{align}
(In fact, $\gamma$ is the same as for the biharmonic operator  \cite{BlumRannacher} and is unique
throughout the paper).  

\subsection{Triangulations and  discrete spaces}\label{subsecTriangulationsanddiscretespaces}
Let $\cT_{\text{init}}$ be a regular  triangulation of the polygonal 
domain $\Omega \subset {\mathbb R}^2$ into triangles in the sense of Ciarlet
\cite{Brenner,Ciarlet}. 
Any refinement is defined by succesive
bisections of refinement edges,  where the refined  triangles inherit  the refinement edges 
according to Figure~\ref{fig1}. The newest vertex bisection (NVB) is described in any space
dimension in \cite{Stevenson08} and is known to generate shape-regular triangulations.  If there exists a finite number of successive bisections 
that start with $\cT_{\text{init}}$ (resp. $\cT$) and end with a regular triangulation $\T$ 
(resp. $\widehat\cT$), then $\T$ is called a 
admissible triangulation   (resp. $\widehat\cT$ is called admissible refinement of $\cT$); 
$\bT=\bT(\cT_{\text{init}})$ is the set of all admissible triangulations and 
$\bT(\cT)$ is the set of all admissible refinements of  $\cT\in\bT$.  In 2D
there are at most $8|\cT_{\text{init}}|$ different interior angles possible in any triangle
$T\in\cT\in\bT$.

This paper concerns two very different subsets of admissible triangulations $\bT$. Given any 
$0<\delta<1$, let $\bT(\delta)$ be the set of all triangulations $\cT$ with 
mesh-size $h_T:=|T|^{1/2}\le \delta$ for all triangles $T\in\cT$ with area $|T|$.
Given any  $N\in\bN_0$, let $\bT(N)$ be the set of all triangulations $\cT$ with at most
$|\cT| \le N+|\cT_{\text{init}}|$ triangles (the counting measure $|\bullet|$ describes the 
cardinality here, but denotes  the Euclidean length or the area at other places).  

\begin{figure}
\begin{center}
\begin{tikzpicture}[scale=0.6]
		
\draw[line width=0.25mm] (0,0)--(2,2)--(4,0);\draw[line width=0.25mm] (4,0)-- (0,0); 
\draw[line width=0.25mm,dashed] (0.1,0.1)--(3.9,0.1);

\draw[line width=0.25mm] (5,0)--(9,0);\draw[line width=0.25mm] (7,2)--(7,0); \draw[line width=0.25mm] (5,0)--(7,2)--(9,0);
\draw[line width=0.25mm,dashed] (5.2,0.1)--(6.9,1.8);\draw[line width=0.25mm,dashed] (7.1,1.8)--(8.9,0);

\draw[line width=0.25mm] (10,0)--(14,0)--(12,2)--(10,0); \draw[line width=0.25mm] (12,2)--(12,0)--(11,1);
\draw[line width=0.25mm,dashed] (10.1,0.1)--(11.9,0.1);
\draw[line width=0.25mm,dashed] (11.9,0.2)--(11.9,1.8);
\draw[line width=0.25mm,dashed] (12.1,1.8)--(13.9,0);

\draw[line width=0.25mm] (15,0)--(19,0)--(17,2)--(15,0); \draw[line width=0.25mm] (17,2)--(17,0)--(18,1);
\draw[line width=0.25mm,dashed] (15.1,0)--(16.9,1.8);
\draw[line width=0.25mm,dashed] (17.1,0.2)--(17.1,1.8);
\draw[line width=0.25mm,dashed] (17.2,0.1)--(18.9,0.1);

\draw[line width=0.25mm] (20,0)--(24,0)--(22,2)--(20,0); \draw[line width=0.25mm] (22,2)--(22,0); \draw[line width=0.25mm] (21,1)--(22,0)--(23,1);
\draw[line width=0.25mm,dashed] (20.1,0.1)--(23.9,0.1);
\draw[line width=0.25mm,dashed] (21.9,0.2)--(21.9,1.8);
\draw[line width=0.25mm,dashed] (22.1,0.2)--(22.1,1.8);
\end{tikzpicture}
\end{center}
\caption{Possible refinements of a triangle $T$ in one level within the NVB. The dashed lines indicate the refinement edges of the sub-triangles as in \cite{BinevDahmenDevore04,Stevenson08}.}
\label{fig1}
\end{figure}
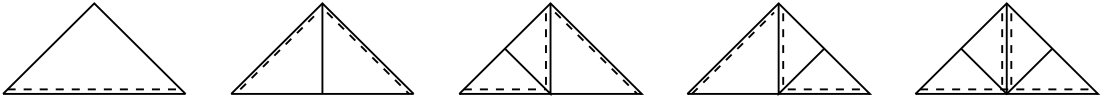

Given any $\cT\in\bT$, let $P_k(\cT)$ denote the piecewise polynomials of degree
at most $k\in \bN_0$. The mesh-size $h_\cT\in P_0(\cT)$ is defined by
$h_\cT|_T:=h_T:= |T|^{1/2} \approx \text{\rm diam}(T)$ in any triangle 
$T\in\cT$ of area $|T|$.
Let   the $L^2$ projection $\Pi_k$ onto the space of piecewise polynomials 
${P}_k(\T)$ of degree at most $k$ act componentwise on vectors or matrices. 
 {The dependence of $\Pi_k$ (and also the interpolation and enhancement operators $I_\M$ and $E_\M$ defined in the next subsection) on the shape regular triangulation $\cT$ is undisplayed for notational convenience.} The oscillations of $f$ in $\cT$ read $\text{osc}_{m}(f,\cT):=\| h_\cT^2(f-\Pi_m f)\|$ for $m\in\bN_0$.  
The associated  nonconforming {\it Morley} finite element space  ${ \M}(\T)$ 
reads 
\[
{\M}(\T):=\Bigg{\{} v_\M\in P_2(\T){{\Bigg |}}
\begin{aligned}
&\; v_\M \text{ is continuous at the interior vertices and vanishes at the}\\
&\text{ vertices of }\partial \Omega;  \nabla_{\NC}{ v_\M} 
                    \text{ is continuous at the midpoints of}\\ 
& \text{ interior edges and vanishes at the  midpoints of boundary edges}\end{aligned}
\Bigg{\}}.
\]
The discrete space in the von K\'{a}rm\'{a}n equations is  $ \bV(\T):=\M(\T) \times \M(\T)$.
For all scalars $ \eta,\chi,\varphi\in  H^2_0(\Omega)+{\M}(\T)$
and all vectors $ \Xi=(\xi_{1},\xi_{2}),$ 
 $\Theta=(\theta_{1},\theta_{2})$, $ \Phi=(\varphi_{1},\varphi_{2})\in {\bf H}^2_0(\Omega)+  \bV(\T)$,  define the discrete bilinear, linear, trilinear  forms by
  \begin{align*}
  & a_\NC(\eta,\chi):=\sit D_{\pw}^2 \eta:D_{\pw}^2\chi\dx\text{ and } b_\NC(\eta,\chi,\varphi):=-\half\sit [\eta,\chi]_{\rm pw} \: \varphi\dx, \\
     & A_\NC(\Theta,\Phi):=a_\NC(\theta_1,\varphi_1)+a_\NC(\theta_2,\varphi_2), \; F(\Phi):=\sit f\varphi_1\dx,  \; \; \text{and}\\
      &B_\NC(\Xi,\Theta,\Phi):=b_\NC(\xi_{1},\theta_{2},\varphi_{1})+b_\NC(\xi_{2},\theta_{1},\varphi_{1})-b_\NC(\xi_{1},\theta_{1},\varphi_{2}). 
\end{align*} 
Notice that $B_\NC(\bullet,\bullet,\bullet)$ is well-defined 
(by the global Sobolev embedding $H^2_0(\Omega) \hookrightarrow L^{\infty}(\Omega)$ {\cite{Adams75}} for the last component)
and symmetric with respect to the first two arguments, i.e., 
$B_\NC(\Xi,\Theta,\Phi)= B_\NC(\Theta,\Xi,\Phi)$ for all $\Xi,\Theta,\Phi\in  \widehat{\bV}$.
The Morley  FEM seeks $\Psi_{\M}=(u_\M,v_\M)\in  \bV(\T)$ such that,  
for all $ \Phi_\M \in  \bV(\T)$,
\begin{equation}\label{vformdNC}
  N_h(\Psi_{\M};\Phi_\M):=A_\NC(\Psi_\M,\Phi_\M)+B_\NC(\Psi_\M,\Psi_\M,\Phi_\M)-F(\Phi_\M)=0.
\end{equation}
\noindent The norm on $\widehat{\bV}:=\bV+\bV(\T)$ is defined by
$\displaystyle \trinl\Phi\trinr_{\NC}:=(\trinl \varphi_{1} \trinr_{\NC}^2+\trinl \varphi_{2} \trinr_{\NC}^2)^{1/2}$ for all $\Phi=(\varphi_1,\varphi_2) \in \widehat{\bV}$ with $ \displaystyle \trinl\varphi_{j}\trinr_{\NC}^2:= a_{\NC}(\varphi_j, \varphi_j)$ 
and  $\displaystyle \| \Phi\|_{1,2,\pw}^2
 :=\|\varphi_{1}\|_{1,2,\pw}^2+\|\varphi_{2}\|_{1,2,\pw}^2$ with 
 $\displaystyle \|\varphi\|_{1,2,\pw}^2 =  \sum_{K \in \T} (\|\varphi\|^2_{2,K}+ \| D \varphi\|_{2,K}^2)  $, 
defines the piecewise $H^1$ norm for  $j=1,2$. 

\subsection{Interpolation and enhancement}\label{sec:int}
Given $\cT\in\bT$ with the maximal mesh-size $h_{\ma}  = \max_{T \in \cT} h_T$ 
and its refinement $\widehat{\cT}\in\bT(\cT)$, define  
the Morley interpolation operator  $I_{\rm M}: \hto+ {\rm M}(\hT) \rightarrow {\rm M}(\T) $  
for any $\wv \in \hto+{\rm M}(\hT)$ through  the degrees of freedom 
\begin{align*} (I_{\rm M} \wv)(z)=\wv(z) \text{ for any vertex $z {\text{ of } \T}$} 
 \quad\text{and}\quad
\int_E\frac{\partial I_{\rm M} \wv}{\partial \nu_E}\ds 
=\int_E\frac{\partial \wv}{\partial \nu_E} \ds\text{ for any edge } E\text{ of } \T .
\end{align*}

\begin{lem}[Morley interpolation \cite{HuShi_Morley_Apost,CCDGJH14, 
DG_Morley_Eigen, CCP}]\label{Morley_Interpolation} 
The Morley interpolation operator satisfies 
 (a) the integral mean property $D^2_{\rm pw} I_{\rm M}v =\Pi_0 D_{\rm pw}^2 v $ for all $v \in H^2_0(\Omega)$ of the Hessian, (b) 
 the approximation and stability property 
\begin{equation*}
h_K^{-4} \|(1-I_{\rm M}) \wv\|^2_{2,K}+h_K^{-2}\|D_{\rm pw}(1-I_{\rm M}) \wv\|^2_{2,K} \lesssim
\|D_{\rm pw}^2 (1-I_{\rm M}) \wv\|^2_{2,K}=\|D_{\rm pw}^2  \wv\|_{2,K}^2-\|D^2_{\rm pw} I_{\rm M} \wv\|_{2,K}^2
\end{equation*}
for $K \in \T$ and $\wv \in \hto+ {\rm M}(\hT)$, and (c)
\[
\|D_{\rm pw}^2 (1-I_{\rm M})v\|_{2} \lesssim h_{\ma}^{\gamma} 
\| v\|_{H^{2+\gamma}(\Omega)} 
\quad\text{for all }\quad v \in \hto \cap H^{2+\gamma}(\Omega). \qed
\]
\end{lem}

Given a vertex $z\in\cN$ in $\cT\in\bT$,  the closure of its patch  $\omega_z:=\text { int } \left( \cup \T(z)\right)$ covers the neighbouring triangles  $\T(z)$  in $\T$ with the vertex $z$. 
Given any triangle $K \in \T$ with its set of  vertices  ${\cal N}(K)$, 
its patch is  $\Omega(K):= \cup_{z \in {\cal N}(K)} \omega_z$   and 
$\mathcal{E}(\Omega(K))$ denotes the set of edges $E$  in $\cT$ with
distance zero to $K$. 

 
\begin{lem}[Companion operator \cite{BrennerMR1620215,DG_Morley_Eigen,CCP}]\label{hctenrich}  
Given any  $\cT\in\bT$ there exists an enrichment or companion operator 
$E_{\rm M}:{\rm M}(\T)\to \hto$ such that any $v_{\rm M} \in \rm M(\T)$ satisfies  
\begin{align*}
&(a)\quad  I_{\rm M} E_{\rm M} v_{\rm M} = v_{\rm M}, \quad (b)\quad 
 \Pi_0(v_{\rm M} - E_{\rm M}  v_{\rm M}) =0,  \quad (c)\quad 
\Pi_0D_{\rm pw}^2(v_{\rm M} - E_{\rm M}  v_{\rm M}) =0,  \\[2mm]
&(d)\quad 
\trinl  v_{\rm M} - E_{\rm M} v_{\rm M}\trinr_{\pw} \lesssim
\|h_\T^{-2} (v_{\rm M}-E_{\rm M} v_{\rm M})\|_{2} 
\lesssim  \min_{v \in \hto} \trinl  v_{\rm M} -v \trinr_{\pw}  ,  \quad\text{and}\\[-3mm]
&(e)\quad \sum_{m=0}^2 h_K^{2m-4} | v_{\rm M}-E_{\rm M} v_{\rm M}|_{H^m(K)}^2 
\lesssim \sum_{E \in \mathcal{E}(\Omega(K))}
 h_E \| [D_{\rm pw}^2v_{\rm M} ]_E \:\tau_E\|^2_{L^2(E)}  \nonumber \\
&\hspace{57.5mm} \lesssim \min_{v \in \hto} \|D^2_{\rm pw}(v_{\rm M} -v)\|^2_{L^2(\Omega(K))}\quad \text{for }K\in\cT. \qquad\qed
\end{align*}
\end{lem}

\begin{rem}[$\widehat I_\M-I_\M=0$ on $\T\cap \hT$]\label{rem:new}
Throughout this paper,  $\widehat{v}_\M - I_\M \widehat{v}_\M=0$ in each $K \in \T\cap\widehat{\T}$ is employed for  $\widehat{v}_\M\in {\rm M}(\hT) $.
(This local projection property follows  from the definition of $I_\M$ for a triangle $K \in \T\cap\widehat{\T}$ and the degrees of freedom in ${\rm M}(\T)$ resp. ${\rm M}(\hT)$.)
\end{rem}

\begin{rem}[consequence] 
It is an immediate consequence of Lemma~\ref{hctenrich} and 
Lemma~\ref{Morley_Interpolation}.c in the end that
\begin{align} \label{magic1}
\trinl I_{\M} {v} - E_{\M} I_{\M} {v} \trinr_{\pw} 
& \lesssim 
\min_{w \in {H^2_0(\Omega)}} \trinl  I_{\M} v   -w \trinr_{\pw} 
 \le \trinl v - I_{\M} v \trinr_{\pw} \lesssim h_{\ma}^{\gamma}
\| v \|_{{H}^{2+\gamma}(\Omega)}.
\end{align}
holds for any $v \in \hto \cap {H}^{2+\gamma}(\Omega)$. 
A similar estimate can be found in \cite[Eq (3.10)]{BSZ2013}.
\end{rem}

\begin{rem}[extension]\label{remexctensions}
The enrichment or a companion operator $E_{\rm M}$ is designed in  \cite{DG_Morley_Eigen} 
 first in terms of the \text{\rm HCT}  FEM
to derive $J_1 v_{\rm M}\in H^2_0(\Omega)$. In the second step,
a linear combination of the squares of the cubic bubble-functions $b_T^2$
on the triangle $T\in\cT$ is added to define $J_2  v_{\rm M}\in H^2_0(\Omega)$ with the 
prescribed integral means to deduce $\Pi_0(   v_{\rm M}-J_2  v_{\rm M})=0$. 
This can be extended to  piecewise polynomials $p_m(T) \, b_T^2$ 
(rather than $\bR \, b_T^2$ for $m=0$) 
to design some $J_{2+m}  v_{\rm M}\in H^2_0(\Omega)$ with 
$\Pi_m(   v_{\rm M}-J_{2+m}  v_{\rm M})=0$ for any $m\in\bN_0$; cf. \cite{ccdgmsMathComp} for  details of a corresponding design for Crouzeix-Raviart
finite element schemes.
\end{rem}

\subsection{Preliminaries}
This subsection collects some preliminary lemmas from earlier 
contributions.  

\begin{lem}[Bounds for $A_{\NC}(\bullet,\bullet)$ {\cite[Lem. 4.2, 4.3]{BSZ2013}\cite[Lem. 4.6]{GMNN_Morley}}] \label{lemmabounds12}
If ${\Phi}, {\boldsymbol\chi}   \in  {\bf H}^{2}_0(\Omega)\cap {\bf H}^{2+\gamma}(\Omega)$ 
and ${\boldsymbol \chi}_{\rm M} \in \bV_{\rm M}$, then  (a)
$A_{\rm pw}({ \Phi}, E_{\rm M} {\boldsymbol \chi}_{\rm M} - {\boldsymbol\chi}_{\rm M}) 
\lesssim h_\ma^{\gamma}  \| {\Phi} \|_{H^{2 + \gamma}(\Omega)} 
\trinl {\boldsymbol \chi}_{\rm M} \trinr_{\rm pw} $
and \\ \mbox{}\hspace{22mm}
(b)   
$A_{\rm pw}({ \Phi}, I_{\rm M} {\boldsymbol \chi} - {\boldsymbol\chi}) 
\lesssim 
h_\ma^{2\gamma}  \|{\Phi} \|_{H^{2 + \gamma}(\Omega)} 
\| {\boldsymbol \chi} \|_{H^{2 + \gamma}(\Omega)}$.  \qed
\end{lem}

The discrete analogs of the global Sobolev embeddings
are of frequent relevance and easily derived with the companions 
for Morley functions; related results are known  \cite[Lem. 3.7]{BNRS}
for $C^0$ functions. 

\begin{lem}[discrete Sobolev embeddings]\label{new12}
Given $\cT_{\text{\rm init}}$ and $1\le p<\infty$,
there exist constants $C_{\text{\rm dea}}, C_{\text{\rm deb}}>0$ such that 
 (a) any  $ \widehat{v} \in H^2_0(\Omega)+{\cal M}(\T)$
satisfies$\|\widehat{v}\|_{\infty} + \|\widehat{v}\|_{1,p,\pw} 
\le C_{\text{\rm dea}}   \trinl \widehat{v} \trinr_{\rm pw}$
 and  (b)  any  $ v_{\rm M}  \in {\cal M}(\cT)\cup H^2_0(\Omega)$  satisfies   $\|{v}_{\rm M}\|_{p}  \le C_{\text{\rm deb}} |{v}_{\rm M}|_{1,2,\pw}$. 
\end{lem}

\begin{proof}
The proof of (a) is included as Lemma~4.7 in \cite{CCGMNN_Semilinear}. 
The proof of (b) for ${v}_\M\in\M(\T)$ 
is based on the following modification of 
Lemma~\ref{hctenrich}.e  for $K\in\hT$ and the shape-regularity in the second inequality below show
\[
  h_K^{-1} \|  {v}_{\rm M}-E_{\rm M} {v}_{\rm M}\|_{2,K}
 +  |  {v}_{\rm M}-E_{\rm M}  {v}_{\rm M}|_{1,2,K}
\lesssim  h_K |  {v}_{\rm M} |_{2,2,\Omega(K),\pw}
\lesssim | h_\cT  {v}_{\rm M} |_{2,2,\Omega(K),\pw}.
 \]
{This}, the finite overlap of the neighbourhoods $(\Omega(K):K\in\hT)$
and a piecewise  inverse estimate in the last step results in 
\begin{equation}\label{eqccnewproofofbnew1}
\|  h_\cT^{-1}(  {v}_{\rm M}-E_{\rm M}  {v}_{\rm M})\|_{2}+
 |  {v}_{\rm M}-E_{\rm M}  {v}_{\rm M}|_{1,2,\pw}\lesssim | h_\cT  {v}_{\rm M} |_{2,2,\pw}
 \lesssim |  {v}_{\rm M} |_{1,2,\pw}.
\end{equation}
%
An inverse estimate  for the finite element function  $v_\M-E_\M v_\M$  \cite{CCP}  shows the first estimate in  
\begin{equation}\label{eqccnewproofofbnew2}
\| {v}_\M- E_\M  {v}_\M\|_{\infty}\lesssim  
\|  h_\cT^{-1}
( {v}_\M- E_\M  {v}_\M)\|_{2}
\lesssim  |  {v}_{\rm M} |_{1,2,\pw}
\end{equation}
and the second follows from \eqref{eqccnewproofofbnew1}.
The continuity of the global Sobolev imbedding $H^1_0(\Omega) \hookrightarrow L^{p}(\Omega)$ {\cite{Adams75}}
shows $\|E_\M  {v}_\M\|_{p}\lesssim \|E_\M  {v}_\M\|_{1,2}\lesssim |E_\M  {v}_\M|_{1,2}$ with the Friedrichs inequality in the second step.
This, triangle inequalities, the  H\"older inequality, and  
\eqref{eqccnewproofofbnew2}  lead to 
\begin{align*}
\| {v}_\M\|_{p} \le \|E_\M  {v}_\M\|_{p} + \|v_\M- E_\M v_\M\|_{p}
\lesssim |E_\M v_\M|_{1,2} +  \|v_\M- E_\M v_\M\|_{\infty}
\lesssim |v_\M|_{1,2,\pw} +  |v_\M- E_\M v_\M|_{1,2,\pw}.
\end{align*}
This and a final application of  \eqref{eqccnewproofofbnew1} conclude the proof of (b).
\end{proof}
{
\begin{rem}
 For $p =\infty $, there would be an additional factor $(1+{|\ln h|})^{1/2}$   multiplied with the constant on the right-hand side of the inequality in Lemma \ref{new12}.b (see   \cite{BS2000}, \cite{ccsc} for details).
\end{rem}}


The following bounds apply frequently in the analysis of this paper; those are  based on 
Lemma~\ref{new12} and the global Sobolev imbedding  {\cite{Adams75}}
 $H^{2+\gamma}(\Omega) \hookrightarrow W^{2,4}(\Omega)$ (the latter requires $\gamma>1/2$).

\begin{lem}[bounds for $B_{\rm pw}(\bullet, \bullet,\bullet)$]\label{b_dG}  
 If $\Phi \in {\bf H}^{2+\gamma}(\Omega)$,  ${ \Theta},
 {\boldsymbol \chi},  {\boldsymbol \zeta}  \in \widehat{\mathcal V} ={\mathcal V} + {\mathcal V}(\cT)$, 
and $ \widehat{{\boldsymbol \chi}}_{\rm M} \in  {\mathcal V}(\hT) $ for $\hT\in\TT{\cT}$, then 
(a) $ B_{\rm pw}( {\boldsymbol \zeta} , \Theta,\boldsymbol \chi) \lesssim
		  \trinl  {\boldsymbol \zeta}  \trinr_{\rm pw}
		  \trinl \Theta \trinr _{\rm pw}\trinl \boldsymbol \chi \trinr_{\rm pw}$
and (b) $B_{\rm pw}(\Phi, \Theta,\widehat{\boldsymbol \chi}_{\rm M})  
 \lesssim 
  \| \Phi\|_{H^{2+\gamma}(\Omega)}\trinl \Theta\trinr_{\rm pw} |\widehat{\boldsymbol \chi}_{\rm M} |_{1,2,\pw}$.
\end{lem}

\begin{proof} (a) The definition of $b_{\NC}(\bullet,\bullet, \bullet)$, piecewise 
H\"older inequalities, and   Lemma~\ref{new12}.a in the last step
show for scalar test functions $\phi, \theta, \chi\in H^2_0(\Omega)+M(\cT)$ that
 \begin{align*}
 2|b_{\NC}(\phi, \theta, \chi)|  = | \sum_{K \in \T} \int_K [\phi, \theta] \chi  \dx | 
 \le \trinl \phi \trinr_{\NC}  \trinl \theta \trinr_{\NC} \|\chi\|_{\infty} 
  \le  C_{\text{\rm dea}} \trinl \phi \trinr_{\NC}  \trinl \theta \trinr_{\NC} \trinl \chi \trinr_{\NC}.
\end{align*} 
The same arguments show in 
(b) with $\phi\in H^2_0(\Omega)\cap H^{2+\gamma}(\Omega)$ and  $\widehat{ \chi}_\M\in M(\hT)$ that
 \begin{align*}
 2|b_{\NC}(\phi , \theta, \widehat{\chi}_\M)|  
  \le \|\phi\|_{2,q} \trinl \theta \trinr_{\NC} \|\widehat{\chi}_\M \|_{p} 
\le C_{\text{\rm S}}C_{\text{\rm deb}}  
\| \phi\|_{H^{2+\gamma}(\Omega)} \trinl \theta \trinr_{\NC} |\widehat{ \chi}_\M |_{1,2,\pw}
\end{align*} 
with the operator norm  $C_S$ of the continuous  global Sobolev imbedding  {\cite{Adams75}}
 $H^{2+\gamma}(\Omega) \hookrightarrow W^{2,q}(\Omega)$ for all $\gamma >0$, $2<q < \infty$, with $\frac{1}{p} +\frac{1}{q}=\frac{1}{2}$
and  Lemma \ref{new12}.b  (with $\hT$ replacing  $\cT$). The application of  the above  estimates in $B(\bullet,\bullet,\bullet)$ concludes the proof.
\end{proof} 
The discrete analog  to  \eqref{inf-sup}
at the regular solution   \cite{CCGMNN_Semilinear}
is a  key result  in the a~priori error analysis.

\begin{thm}[discrete inf-sup \cite{CCGMNN_Semilinear}]\label{dis_stab} 
Given a regular solution $\Psi \in {\bf H}^{2 + \gamma}(\Omega) \cap {\bV}$ to \eqref{vform_cts},   
there exist $\delta_1 >0$ and $\beta_1>0$ such that $\T \in \TO(\delta_1)$ implies 
 \begin{align}\label{dis-inf-sup}
\beta_1\le \inf_{\substack{\Theta_{\rm M} \in \rm \bV(\T) \\ \trinl\Theta_{\rm M}\trinr_{\rm pw}=1}}
\sup_{\substack{\Phi_{\rm M} \in \rm \bV(\T)\\ \trinl\Phi_{\rm M}\trinr_{\rm pw}=1}}
\big{(}A_{\rm pw}(\Theta_{\rm M},\Phi_{\rm M})+2B_{\rm pw}(\Psi,\Theta_{\rm M},\Phi_{\rm M})\big{)}.
\qquad\qed
\end{align} 
\end{thm}

\section{A priori error analysis}
The {\it a priori} energy norm estimates for the Morley FEM for \eqref{vform_cts} can be found in  \cite{CCGMNN_Semilinear, GMNN_Morley}. The subsequent theorem  adds 
a new piecewise $H^1$ semi-norm {\it a priori}  error estimate. 

\begin{thm}[a~priori]\label{aprioriestimate} Given a  regular solution 
$\Psi$  to \eqref{vform_cts}, there exist $\varepsilon_0, \delta_0>0$ 
(without loss of generality $\delta_0<1$)
such that, for all $\T \in \TO(\delta_0)$, 
(a)  there exists a unique solution $\Psi_{\rm M} \in \bV(\T)$  
to \eqref{vformdNC} with $\trinl \Psi-\Psi_{\rm M}\trinr_{\rm pw} \le \varepsilon_0$;
the solutions $\Psi=(u,v)$ and $\Psi_{\rm M}$  satisfy 
\[ 
(b) \;\trinl \Psi-\Psi_{\rm M}\trinr_{\rm pw} \lesssim 
\trinl \Psi-I_{\rm M}\Psi \trinr_{\rm pw} + \text{\rm osc}_1(f+[u,v],\cT)+
\text{\rm osc}_1([u,u],\cT),  
\] 
and (c)
\(
| \Psi-\Psi_{\rm M}|_{1,2,\pw} 
\lesssim h_{\max}^{ \gamma} 
(\trinl \Psi-\Psi_{\rm M}\trinr_{\rm pw} +{\rm osc}_m(f, \cT)) 
\)
 for each $m\in\bN_0$. 
\end{thm}
The parts (a) on the existence of $\Psi_\M \in \bV(\T)$ and (b) on 
the best-approximation in the energy norm are  established in { \cite[Theorems 3.3, 5.3]{CCGMNN_Semilinear}, \cite{GMNN_Morley}.
The proof of the piecewise $H^1$ error estimate in (c)  depends on 
the following lemma about the nonlinearity and  the companion
 ${E}_{\rm M}{\boldsymbol \rho}_{\rm M}$ from  Lemma \ref{hctenrich} (where ${E}_{\rm M}$ acts componentwise).

\begin{lem}\label{apriorinew} If 
$\Psi , {\boldsymbol \zeta} \in {\bH}^{2+\gamma}(\Omega) \cap {\mathcal V}$   
and  ${\boldsymbol \rho}_{\rm M} \in \bV(\T)$, then
(a) $ B_{\rm pw}(\Psi, {E_{\rm M}} {\boldsymbol \rho}_{\rm M} 
- {\boldsymbol \rho}_{\rm M}, {\boldsymbol \zeta} )  \lesssim  h^{\gamma}_\ma  
\| \Psi \|_{{H}^{2+\gamma}(\Omega)} $ $\times  \trinl {\boldsymbol \rho}_{\rm M}  \trinr_{\pw}  
\| {\boldsymbol \zeta} \|_{{H}^{2+\gamma}(\Omega)}$ and 
(b) $B_{\rm pw}(\Psi, {I}_{\rm M} \Psi - \Psi, {\boldsymbol \zeta} ) 
\lesssim   h^{\gamma}_\ma 
 \| \Psi \|_{H^{2+\gamma}(\Omega)}  \trinl \Psi -\Psi_{\rm M}\trinr_{\rm pw} 
 \| {\boldsymbol \zeta} \|_{{H}^{2+\gamma}(\Omega)}$.
\end{lem}

\begin{proof}
The higher smoothness of the test functions in the first and last arguments 
$ {\Psi}=(\psi_1,\psi_2), {\boldsymbol \zeta}
=(\zeta_1,\zeta_2)\in {\bH}^{2+\gamma}(\Omega) \cap {\mathcal V}$   
in   $B_{\NC}(\Psi, {\boldsymbol \chi}  , {\boldsymbol \zeta} )$ 
is combined with an orthogonality for the middle test function
$ {\boldsymbol \chi} := {E_{\M}} {\boldsymbol \rho}_{\M} - {\boldsymbol \rho}_{\M}$
in  (a) and   ${\boldsymbol \chi} :=\Psi-I_\M\Psi$ in  (b) to allow an extra factor $h^{\gamma}_\ma $.
Lemma~\ref{hctenrich}.c in (a) (resp. Lemma~\ref{Morley_Interpolation}.a in (b))
proves that  $ {\boldsymbol \chi}=(\chi_1,\chi_2)$  has a  piecewise  second order derivative 
$D^2_\text{\rm pw} {\boldsymbol \chi} \perp P_0(\cT;\bR^{2\times 2})$ with integral mean
zero over each triangle in $\cT$. To exploit the consequences, observe that
$B_{\NC}(\Psi, {\boldsymbol \chi}  , {\boldsymbol \zeta} )$  is a sum
of terms of the form 
\begin{align*}
I:= \int_\Omega  ( \zeta_k  \partial_{ab}  \psi_j )\,  \partial_{cd,\pw} \: \chi_\ell  \dx &
=   \int_\Omega ( \lambda _{kabj} -\Pi_0  \lambda _{kabj}  )  \partial_{cd,\pw}\chi_\ell  \dx
\le F_1F_2 
\end{align*}   
for some indices  $a,b,c,d, j,k,\ell \in \{1,2\}$ and the abbreviation
$ \lambda _{kabj}:=  \zeta_k  \partial_{ab}  \psi_j\in H^\gamma(\Omega)$
with piecewise integral means $ \Pi_0\lambda _{kabj}$
(the undisplayed sign does not play a role in the estimates below).  The
integral $I$ allows for a Cauchy inequality with one factor 
\[
F_1:= \| \lambda _{kabj}  -\Pi_0  \lambda _{kabj}  \|_2  
\le  \|\zeta_k \: \partial_{ab}  \psi_j -  \Pi_0 \zeta_k   \: \Pi_0\partial_{ab}  \psi_j \|_2 
\]
(with the inequality from the $L^2$ orthogonality). A triangle inequality,   
$\|\Pi_0\zeta_k\|_\infty\le  \|\zeta_k\|_\infty$, and 
$\|\zeta_k-\Pi_0\zeta_k\|_\infty \le h_{\max} |\zeta_k|_{1,\infty}$ show  
\begin{align*}
F_1 & \le 
\|(\zeta_k-\Pi_0 \zeta_k ) \: \partial_{ab}\psi_j \|_2 
+ \|\zeta_k\|_\infty \| \partial_{ab} \psi_j -\: \Pi_0\partial_{ab}  \psi_j\|_2 \\
& \lesssim  \| \zeta_k \|_{W^{1,\infty}(\Omega)} (h_\ma \| \psi_j\|_{H^{2}(\Omega)}+ h^{\gamma}_\ma\|\psi_j\|_{H^{2+\gamma}(\Omega)})
\end{align*}
with a well-known estimate 
$
\|\varphi- \Pi_0 \varphi\|_2\lesssim h_\ma^{\gamma}\|\varphi\|_{H^\gamma(\Omega)}
$ 
for $\varphi:= \partial_{ab}  \psi_j \in H^\gamma(\cT)$  \cite{Brenner} in the last step.
This and the global Sobolev embedding  {\cite{Adams75}}  $H^{2+\gamma}(\Omega) \hookrightarrow 
W^{1,\infty}(\Omega)$ 
conclude the analysis for the first factor in the upper bound of the prototype term 
\[
F_1  \lesssim  h^{\gamma}_\ma \|\Psi\|_{H^{2+\gamma}(\Omega)}  
\|{\boldsymbol \zeta}\|_{H^{2+\gamma}(\Omega)}  .
\]
In case (a), the second factor  
$F_2:= \| \partial_{cd,\pw}\chi_\ell \| \le   \trinl   {\boldsymbol \chi}  \trinr_{\pw} 
=  \trinl   {E_{\M}} {\boldsymbol \rho}_{\M} - {\boldsymbol \rho}_{\M}  \trinr_{\pw} $ 
is controlled with  Lemma~\ref{hctenrich}.d and  $v=0$ as   
$F_2\lesssim \trinl   {\boldsymbol \rho}_{\M}  \trinr_{\pw} $.
In the other case (b), $F_2 \le   \trinl   {\boldsymbol \chi}  \trinr_{\pw}  =
 \trinl  \Psi-I_\M\Psi \trinr_{\pw} $.
Those two estimates and the previously displayed estimate for $F_1$
result in an upper bound of $F_1F_2\ge |I|$. {The Pythagoras theorem 
from the best-approximation 
property of $I_\M{\Psi}$ to $\Psi$ in $\bV(\T)$ of Lemma~\ref{Morley_Interpolation}.a verifies 
$ \trinl  \Psi-I_\M\Psi \trinr_{\pw}  \le  \trinl  \Psi- \Psi_\M \trinr_{\pw}. $}
The evaluation of all those contributions of type $I$ 
concludes the 
proof of  (a) and (b).
\end{proof}

\noindent{\it Proof of Theorem~\ref{aprioriestimate}.c.} 
The proof is by  a careful analysis of the perturbations from the
nonconforming functions  with the companion operators in a duality argument.  
The {\it first step} is the definition and the isolation 
of the crucial term ${E_\M} { \boldsymbol \rho}_{\M}\in \bV$: 
Let ${\boldsymbol \rho}_{\M} := I_{\M}\Psi -\Psi_{\M} \in \bV(\T)$ and 
recall  ${E}_{\M}$  of  Lemma~\ref{hctenrich} (acting componentwise). 
Triangle inequalities show that 
\begin{align} \label{triangle1}
\| \Psi-\Psi_{\M}\|_{1,2,\pw} \le \| \Psi-I_{\M}\Psi \|_{1,2,\pw} + \| {\boldsymbol \rho}_{\M} -{E}_{\M}  {\boldsymbol\rho}_{\M}\|_{1,2,\pw} + \| {E_\M} { \boldsymbol \rho}_{\M} \|_{1,2,\pw}.
\end{align}
Lemma~\ref{Morley_Interpolation}.b is followed by  the Pythagoras theorem 
from the best-approximation 
property of $I_\M{\Psi}$ to $\Psi$ in $\bV(\T)$ of Lemma~\ref{Morley_Interpolation}.a
to verify
\begin{align} \label{ft}
h_{\rm max}^{-2} \| \Psi-I_{\M}\Psi \|_{1,2,\pw}^2 \lesssim  \trinl\Psi-I_{\M} \Psi\trinr_{\NC} ^2
=  \trinl \Psi-\Psi_{\M} \trinr_{\NC}^2- \trinl   { \boldsymbol \rho}_{\M} \trinr_{\NC}^2.
\end{align}
Lemma~\ref{hctenrich}.a shows that 
${\boldsymbol \rho}_{\M} -{E}_{\M} {\boldsymbol\rho}_{\M}= (I_\M-1)
{E}_{\M} {\boldsymbol\rho}_{\M}$ and so 
{Lemma~\ref{Morley_Interpolation}.b (in a global version)
proves the first inequality in 
\begin{align} \label{st}
h_{\rm max}^{-1}\| {\boldsymbol \rho}_{\M} -{E}_{\M}  {\boldsymbol\rho}_{\M}\|_{1,2,\pw} 
\lesssim \trinl {\boldsymbol \rho}_\M -{E}_{\M}   {\boldsymbol\rho}_{\M} \trinr_{\NC} \lesssim
\trinl {\boldsymbol \rho}_\M \trinr_{\NC}  & 
 \le  \trinl \Psi-\Psi_{\M} \trinr_{\NC}.
\end{align}
 The second inequality in \eqref{st} is the stability  from 
Lemma~\ref{hctenrich}.d with $v=0$.}
The last inequality  in \eqref{st}
follows from the Pythagoras theorem in  \eqref{ft}. This concludes the first step 
in which \eqref{triangle1}-\eqref{st} show
\begin{align} \label{intermediateeqst}
\| \Psi-\Psi_{\M}\|_{1,2,\pw} \lesssim h_{\rm max}  \trinl \Psi-\Psi_{\M} \trinr_{\NC}
+ \| {E_\M} { \boldsymbol \rho}_{\M} \|_{1,2,\pw}.
\end{align}
The {\it second step} focuses on the last term in \eqref{intermediateeqst}
with a duality argument with 
the solution ${\boldsymbol \zeta} \in \bV$ to the linearized equation 
\eqref{duallinearized} where 
${\boldsymbol g}:=-\Delta {E}_{\M} {\boldsymbol \rho}_{\M} \in {\bf L}^2(\Omega)$ 
(${E}_{\M}$ acts componentwise on ${\boldsymbol \rho}_{\M} $ and 
the Laplacian  acts componentwise on $ {E}_{\M} {\boldsymbol \rho}_{\M} \in {\bf H}^2_0(\Omega)$)
and the regularity  
${\boldsymbol \zeta}  \in {\bf H}^{2+\gamma}(\Omega)$  from \eqref{regularity}. 
Recall  the linearised operator ${\mathcal A}(\bullet,\bullet)$ 
of \eqref{duallinearized} and its piecewise analog ${\mathcal A}_{\NC}(\bullet,\bullet)$ that replaces $A(\bullet,\bullet)$ (resp. $B(\bullet,\bullet,\bullet)$) in \eqref{duallinearized} by $A_{\NC}(\bullet,\bullet)$  (resp.  $B_{\NC}(\bullet,\bullet,\bullet)$).
This and  elementary   algebra eventually
lead to 
\begin{align*} \label{intt1}
& | {E}_{\M}\rho_{\M} |_{1,2}^2  = (\nabla {E}_{\M} {\boldsymbol \rho}_{\M}, \nabla {E}_{\M} {\boldsymbol \rho}_{\M})_{{\bf L}^2(\Omega)} 
= ({\boldsymbol g} , {E}_{\M} {\boldsymbol \rho}_{\M})_{{\bf L}^2(\Omega)}
 = {\mathcal A}({E}_{\M} {\boldsymbol \rho}_{\M}, {\boldsymbol \zeta}) \nonumber \\
 & \quad= A_{\NC}({E}_{\M} {\boldsymbol \rho}_{\M} - {\boldsymbol \rho}_{\M}, {\boldsymbol \zeta} ) + 2B_{\NC}(\Psi, {E}_{\M} {\boldsymbol \rho}_{\M} - {\boldsymbol \rho}_{\M}, {\boldsymbol \zeta} ) + {A}_{\NC}(I_{\M} \Psi- \Psi, {\boldsymbol \zeta})  \nonumber \\
&  \quad\quad +  2B_{\NC} (\Psi,I_{\M} \Psi- \Psi, {\boldsymbol \zeta} ) + {A}_{\NC}(\Psi-\Psi_{\M}, {\boldsymbol \zeta})
+ 2B_{\NC} (\Psi, \Psi-\Psi_{\M},{\boldsymbol \zeta}) \nonumber \\
 & \quad =  A_{\NC}({E}_{\M} {\boldsymbol \rho}_{\M} - {\boldsymbol \rho}_{\M}, {\boldsymbol \zeta} ) + 2B_{\NC}(\Psi, {E}_{\M} {\boldsymbol \rho}_{\M} - {\boldsymbol \rho}_{\M}, {\boldsymbol \zeta} ) + 2B_{\NC} (\Psi,I_{\M} \Psi- \Psi, {\boldsymbol \zeta} ) \nonumber \\
 & \quad  \quad + {A}_{\NC}(I_{\M} \Psi- \Psi, {\boldsymbol \zeta}) +A_\NC(\Psi-\Psi_\M, {\boldsymbol \zeta} -  I_\M {\boldsymbol \zeta}) +A_\NC(\Psi-\Psi_\M, I_\M{\boldsymbol \zeta} - E_\M I_\M {\boldsymbol \zeta})  \nonumber \\
 & \quad \quad + A_\NC(\Psi-\Psi_\M,  E_\M I_\M {\boldsymbol \zeta})  
   + 2B_{\NC} (\Psi, \Psi-\Psi_{\M},{\boldsymbol \zeta})=: T_1+\dots+T_8.
\end{align*}
The eight terms  $T_1,\cdots,T_8$ are controlled in {\it step three} of the proof.  
Lemma~\ref{lemmabounds12}.a shows 
\begin{align*} 
T_1  := A_{\NC}({E}_{\M} {\boldsymbol \rho}_{\M}- {\boldsymbol \rho}_{\M}, 
{\boldsymbol \zeta}) 
\lesssim h_\ma^{\gamma} \trinl {\boldsymbol \rho}_{\M} \trinr_{\NC}
\| {\boldsymbol \zeta} \|_{{H}^{2 + \gamma}(\Omega)}
\lesssim h_\ma^{\gamma} \trinl  \Psi -\Psi_\M\trinr_{\NC}
\| {\boldsymbol \zeta} \|_{{H}^{2 + \gamma}(\Omega)}
 \end{align*}
 with   \eqref{st} in the end.
 Lemma~\ref{apriorinew} and  \eqref{st}  imply
\begin{align*} 
T_2+T_3&= 2B_{\NC}(\Psi, {E}_{\M} {\boldsymbol \rho}_\M  
-{\boldsymbol \rho}_\M , {\boldsymbol \zeta} ) 
+2B_{\NC}(\Psi, {I}_{\M} \Psi - \Psi, {\boldsymbol \zeta} ) 
\\ & \lesssim  h^{\gamma}_\ma  \| \Psi \|_{{H}^{2+\gamma}(\Omega)}  
\trinl \Psi -\Psi_\M\trinr_\NC 
\| {\boldsymbol \zeta} \|_{{H}^{2 + \gamma}(\Omega)}.
\end{align*}
Lemma~\ref{Morley_Interpolation}.a 
shows
$A_{\NC}(\Psi-I_\M\Psi, {\boldsymbol \eta}_\M)=0
= A_{\NC}({\boldsymbol \eta}_\M, {\boldsymbol \zeta}-I_\M {\boldsymbol \zeta})$ 
for all ${\boldsymbol \eta}_\M \in {\bV(\T)}$. Consequently,
\begin{align*}
T_4+T_5  &={A}_{\NC}(I_{\M} \Psi- \Psi, {\boldsymbol \zeta}-I_\M {\boldsymbol \zeta}) +  {A}_{\NC}(\Psi-\Psi_{\M}, {\boldsymbol \zeta}-I_{\M} {\boldsymbol \zeta})=0.
\end{align*}
The boundedness of $A_{\pw}(\bullet,\bullet)$ and  \eqref{magic1} result in 
\[
T_6 :=A_\NC(\Psi-\Psi_\M,  I_\M {\boldsymbol \zeta}-E_\M I_\M {\boldsymbol \zeta}) 
\lesssim h_{\ma}^{\gamma} \trinl \Psi- \Psi_\M\trinr_{\pw}
\|{\boldsymbol \zeta} \|_{{H}^{2+\gamma}(\Omega)} .
\]
Lemma \ref{hctenrich}.a shows  $I_{\M} E_{\M} \varphi_{\M}=\varphi_{\M}$ for $\varphi_\M \in \M(\T)$ and  
Lemma \ref{hctenrich}.c shows $A_{\NC}(\Psi_\M, (1-I_\M)E_\M I_\M {\boldsymbol \zeta}) =0$.  This and
\eqref{vform_cts}  (resp. \eqref{vformdNC}) in the end lead to  
\begin{align*} \label{interme1}
T_7& :=A_\NC(\Psi-\Psi_\M,  E_\M I_\M {\boldsymbol \zeta})  
= A_\NC(\Psi,  E_\M I_\M {\boldsymbol \zeta})-  A_\NC(\Psi_\M,  I_\M {\boldsymbol \zeta}) \nonumber \\&  =
 F(E_\M I_\M {\boldsymbol \zeta} - I_\M {\boldsymbol \zeta} )- B_{\pw}(\Psi,\Psi, E_\M I_\M {\boldsymbol \zeta}) 
+ B_{\pw}(\Psi_\M, \Psi_\M, I_\M {\boldsymbol \zeta}).
\end{align*}
Lemma~\ref{hctenrich}.b shows  $\Pi_0\Theta =0 $ for 
 $\Theta := E_\M I_\M {\boldsymbol \zeta}-I_\M {\boldsymbol \zeta} 
=(1-I_\M ) E_\M I_\M {\boldsymbol \zeta}$.
Since the piecewise second derivatives of $\Psi_\M$ are piecewise constants, 
$B_{\pw}(\Psi_\M, \Psi_\M, \Theta)=0$ follows.
This and elementary algebra (with the symmetry of $B_{\pw}$ in the first two components) 
imply 
\[
T_7+T_8=
 F(\Theta)
+ B_{\NC} (\Psi-\Psi_{\M}, \Psi-\Psi_{\M}, E_\M I_\M{\boldsymbol \zeta} ) 
+2 B_{\pw}(\Psi-\Psi_\M, \Psi,  {\boldsymbol \zeta} -E_\M I_\M {\boldsymbol \zeta}).
\]
The right-hand side consists of three terms $ T_9+T_{10}+T_{11}$  estimated in the sequel.
Recall $\Pi_0\Theta =0 $  
and apply 
Lemma \ref{Morley_Interpolation}.b to verify {for $\Theta=(\Theta_1,\Theta_2)$,}
\begin{align*}
T_9:= F(\Theta )= (f-\Pi_0f, \Theta_1)_{L^2(\Omega)} 
 \lesssim  {\rm osc}_0(f,\T)  \trinl \Theta \trinr_{\NC} 
  \lesssim h_\ma^{\gamma} {\rm osc}_0(f,\cT)    \| {\boldsymbol \zeta}\|_{H^{2+\gamma}(\Omega)} 
\end{align*}
with  $\trinl \Theta  \trinr_{\NC} \lesssim 
 h_\ma^{\gamma} \| {\boldsymbol \zeta}\|_{H^{2+\gamma}(\Omega)} $ from
\eqref{magic1} in the final step.
Lemma \ref{b_dG}.a and  Theorem \ref{aprioriestimate}.b show
\begin{align*}
T_{10}:=B_{\NC} (\Psi-\Psi_{\M}, \Psi-\Psi_{\M}, E_\M I_\M{\boldsymbol \zeta}) 
& \lesssim h_\ma^{\gamma}
 \| \Psi \|_{H^{2+\gamma}(\Omega)}\trinl \Psi-\Psi_\M \trinr_{\NC} 
 \trinl E_\M I_\M{\boldsymbol \zeta}  \trinr_{\NC} .
\end{align*}
The stability $ \trinl E_\M I_\M{\boldsymbol \zeta}  \trinr_{\NC} 
\lesssim  \trinl  I_\M{\boldsymbol \zeta}  \trinr_{\NC} 
\lesssim  \trinl  {\boldsymbol \zeta}  \trinr$
from  Lemma~\ref{hctenrich}.d  proves 
$ \trinl E_\M I_\M{\boldsymbol \zeta}  \trinr_{\NC} 
\lesssim  \| {\boldsymbol \zeta} \|_{H^{2+\gamma}(\Omega)} $.
{Recall that  $B_{\NC} (\bullet,\bullet,\bullet)$ is symmetric with respect to the first two arguments}. 
For $  {\boldsymbol \zeta} -E_\M I_\M {\boldsymbol \zeta}\in \cV$, \cite[Lemma 3.9b]{CCGMNN18} shows that}
\[
T_{11}=
 2 B_{\pw}(\Psi, \Psi-\Psi_\M,   {\boldsymbol \zeta} -E_\M I_\M {\boldsymbol \zeta}) 
  \lesssim \| \Psi\|_{H^{2+\gamma}(\Omega)} 
\trinl  \Psi-\Psi_\M \trinr_\NC 
\|   {\boldsymbol \zeta}- E_\M I_\M {\boldsymbol \zeta}\|_{1,2,\pw}.
\]
A triangle inequality, Lemma~\ref{Morley_Interpolation}.b-c, {Lemma \ref{hctenrich}.e}, and \eqref{magic1}
lead to 
\[
\|   {\boldsymbol \zeta}- E_\M I_\M {\boldsymbol \zeta}\|_{1,2,\pw}
\le  \|   {\boldsymbol \zeta}-I_\M {\boldsymbol \zeta}\|_{1,2,\pw}
+\|  I_\M  {\boldsymbol \zeta}- E_\M I_\M {\boldsymbol \zeta}\|_{1,2,\pw} 
\le h_\ma^{1+\gamma}    \| {\boldsymbol \zeta}\|_{H^{2+\gamma}(\Omega)} .
\]
The combination of the aforementioned estimates with 
$1/2<\gamma\le 1$ results in  
\[
 T_7+T_8=T_9+T_{10}+T_{11}
 \lesssim 
 h_\ma^{\gamma} \left(  \trinl \Psi-\Psi_\M \trinr_{\NC}  + {\rm osc}_0(f,\cT)\right)   
 \| {\boldsymbol \zeta}\|_{H^{2+\gamma}(\Omega)}.
\]
{\it Step four} is the conclusion of the proof. 
The  estimates for $T_1$ to $T_8$ 
lead to 
\[
 \| {E}_{\M}\rho_{\M}\|_{1,2}^2 \lesssim
 h^{\gamma}_\ma  (1+\| \Psi \|_{{H}^{2+\gamma}(\Omega)} )  
\left( \trinl \Psi -\Psi_\M\trinr_\NC  + {\rm osc}_0(f, \T)\right) 
\| {\boldsymbol \zeta} \|_{{H}^{2 + \gamma}(\Omega)}.
\]
Recall \eqref{regularity} with 
$
\| {\boldsymbol \zeta}\|_{H^{2+\gamma}(\Omega)}
 \lesssim   \| {\boldsymbol g}\|_{-1}
 \lesssim  \|  {E}_{\M} {\boldsymbol \rho}_{\M} \|_{1,2}$, so that
\[
 \| {E}_{\M}\rho_{\M}\|_{1,2} \lesssim
 h^{\gamma}_\ma  (1+\| \Psi \|_{{H}^{2+\gamma}(\Omega)} )  
\left( \trinl \Psi -\Psi_\M\trinr_\NC  + {\rm osc}_0(f, \T)\right) .
\]
Recall Theorem~\ref{ap} and write  $ \| \Psi \|_{{H}^{2+\gamma}(\Omega)}  \lesssim 1$ 
(so the constants depend on $\| f\|_{-1} $).  This, the  previous estimate,  and 
 \eqref{intermediateeqst} conclude the proof for $m=0$. For $m\in\bN$,  utilise a companion operator 
$E_{\rm M}$ outlined in  Remark~\ref{remexctensions} with all the properties
of Lemma~\ref{hctenrich} plus the higher-order orthogonality
$ \Pi_m(v_{\rm M} - E_{\rm M}  v_{\rm M}) =0$  for all $v_{\rm M} \in \rm M(\T)$.
This allows in $T_9$ the extra orthogonality
\[
F(\Theta )= (f-\Pi_mf, \Theta_1)_{L^2(\Omega)} 
 \lesssim  {\rm osc}_m(f,\T)  \trinl \Theta \trinr_{\NC} 
  \lesssim h_\ma^{\gamma} {\rm osc}_m(f,\cT)    
  \| {\boldsymbol \zeta}\|_{{H}^{2+\gamma}(\Omega)} .
\]
This modification enables the proof for general $m\in\bN $; further details are omitted.
\qed

\section{Adaptive mesh-refinement and the axioms of adaptivity}
\label{secAdaptivemesh-refinementandtheaxiomsofadaptivity}
In the remainder of this paper,  $\Psi=(u,v)\in \bV$ is a fixed 
regular solution to \eqref{vform_cts} called  {\em the exact solution}.
Theorem~\ref{aprioriestimate}.a leads to $\varepsilon_0, \delta_0>0$ 
such that any triangulation $\cT\in\bT(\delta_0)$
leads to a  unique discrete solution $\Psi_{\rm M} \in \bV(\T)$  
to \eqref{vformdNC} with $\trinl \Psi-\Psi_{\rm M}\trinr_{\rm pw} \le \varepsilon_0$
and this $\Psi_{\rm M}=(u_{\rm M},v_{\rm M})$ is called {\em the discrete solution}.

\subsection{A posteriori error analysis} 
Given  the  discrete solution $\Psi_{\rm M}=(u_{\rm M},v_{\rm M}) \in \bV(\T)$   to 
 \eqref{vformdNC} 
define  $ \eta(\T,K) \ge 0 $ as the square root of 
\begin{align}
\label{estimators} 
\eta^2(\T,K) & := |K|^2\left( \|[u_\M,v_\M]+f \|_{2,K}^2+\|[u_\M,u_\M]\|_{2,K}^2\right) 
\nonumber \\
&  + |K|^{1/2} \sum_{E \in {\cal E}(K) } \left( 
\|[D_\NC^2 u_\M]_E\tau_E\|_{L^2(E)}^2+\|[D_\NC^2 v_\M]_E\tau_E \|_{L^2(E)}^2 \right).
\end{align}
Recall the notation from  Subsection~\ref{subsecTriangulationsanddiscretespaces}
and let  ${\cal E}(K)$ denote  the three edges of a triangle $K\in\cT$ with area $|K|$. 
The jump $[\bullet]_E$ across an interior edge 
$E=\partial T_+ \cap \partial T_-\in \cE(\Omega)$
with tangential vector $\tau_E$ and normal $\nu_E$ 
is the difference of the respective traces on $E$ from the two neighbouring 
triangles $T_{\pm}$ that form the edge
patch $\omega_E:=\text{int}(T_+\cup T_-)$. 
The jump  $[\bullet]_E$ along  a  boundary edge $E\in \cE(\partial\Omega)$
is simply the trace from the attached triangle $T_+=\overline{\omega_E}$;
the contribution of the missing jump partner is zero.
Like any other operator,  $[\bullet]_E$  acts componentwise
in  $ \|[D_\NC^2 u_\M]_E\tau_E\|_{L^2(E)}$ {and $ \|[D_\NC^2 v_\M]_E\tau_E\|_{L^2(E)}$}. 
Given  any subset $\mM\subseteq \T$ of $\cT\in\bT(\delta_0)$, its contribution
 $\eta(\T,{\mathcal M})\ge 0 $  is the square root of the sum 
\begin{align}\label{eq:def_estim}
\eta^2(\T,{\mathcal M}):=\sum_{K\in\mM}\eta^2(\T,K)
\quad\text{and } \eta(\cT):=\eta(\cT,\cT) 
\end{align}
abbreviates the contribution of all triangles  (by convention 
$\eta(\T,\emptyset):=0$). Recall  the oscillations 
${\rm osc}_0^2(f,\cT )= \|h^2_{\T}(1-\Pi_0)f\|^2_{L^2(\Omega)}$ of $f\in L^2(\Omega) $ for the 
mesh-size factor 
 $h_\T \in P_0(\T)$. 

\begin{thm}[a posteriori \cite{CCGMNN_Semilinear}]
\label{reliability_NC} 
Given the exact solution $\Psi$ and $\varepsilon_0$ and $\delta_0$ from
Theorem~\ref{aprioriestimate}.a, there exist positive constants $C_{\rm rel}$ and $C_{\rm eff} $
(which depend on $\cT_{\rm init}$ and on $\Psi$, $\varepsilon_0$, $\delta_0$) so that, for all $\cT\in\bT(\delta_0)$,
 the discrete solution $\Psi_{\rm M} \in \bV(\T)$  and the error estimator $\eta(\cT)$  from
 \eqref{estimators}- \eqref{eq:def_estim}
satisfy 
\begin{align*}
C^{-1}_{\rm rel}\trinl{\boldmath\Psi}-{\boldmath\Psi}_{\rm M}\trinr_{\rm pw}&\leq \eta(\cT) 
	\leq C_{\rm eff} \left( \trinl{\boldmath \Psi}-{\boldmath\Psi}_{\rm M}\trinr_{\rm pw}
	+{\rm osc}_0(f,\cT)\right). \qquad\qquad\qed
\end{align*}
\end{thm}

The proofs can be found in   \cite{CCGMNN_Semilinear}; an alternative proof of the reliability 
(the first inequality in Theorem~\ref{reliability_NC}) follows 
in Corollary~\ref{eq:def_estim} below. This paper does not exploit the  efficiency (the second inequality in Theorem \ref{reliability_NC}) except for Remark~\ref{rem:nonlinearapproximation}.
 
\subsection{Adaptive Morley finite element algorithm}
Recall that the initial  triangulation $\cT_{\text{init}}$ satisfies 
the initial condition (IC) for its reference edges from 
Subsection~\ref{subsecTriangulationsanddiscretespaces}.
Recall  ${\boldmath\Psi}$ and $\varepsilon_0, \delta_0>0$ from Theorem~\ref{aprioriestimate}.a.
\\[3mm]
{\bf Adaptive algorithm (AMFEM).}\\
\begin{algof}
	\DontPrintSemicolon
	\KwIn{Initial  triangulation $\cT_{\text{init}}$ with IC,  
	$0<\delta\le \delta_0<1$, and 	$0<\theta\le  1$ 
}
	\acompute{ $\cT_0$ by uniform refinements of $\cT_{\text{init}}$ such that
	$\cT_0\in\bT(\delta)$}

	\For{$\ell=0,1,\ldots$}	
	{
		\acompute{ discrete solution ${\boldmath\Psi}_{\ell}=(u_\M,v_\M) \in \bV(\T_{\ell})$ 
	         with $\trinl {\boldmath\Psi}-{\boldmath\Psi}_\ell \trinr_{\NC} \le \varepsilon_0$	
	         
		\acompute{
		$\eta_{\ell}(K):=\eta(\cT_\ell,K) $  
		by  \eqref{estimators} for all $K \in \Tl \ell$   }}
			{\vspace{-2ex}
		\aselect{ $\Ml \ell \subseteq \Tl \ell$ of (almost) minimal cardinality with 
		\vspace{-1.5ex}
		\begin{align}\label{sAfem:eq:bulkA}
			\theta \: \eta^2_\ell (\cM_\ell ) \leq \eta_\ell^2 := \eta^2(\T_\ell)
		\end{align}}
		\vspace{-5ex}
		\acompute{$\Tl {\ell+1}:= \textsc{Refine}(\Tl \ell,\Ml \ell)$ \vspace{0.5ex}}
		}
		}
		
		\KwOut{$\Tl \ell$, ${\boldmath\Psi}_\ell$, and  $\eta_\ell$ for $\ell\in \bN_0$}
\end{algof}

Throughout this paper, $\Tl \ell$, ${\boldmath\Psi}_\ell$, and  $\eta_\ell$ will refer to the 
output of this adaptive algorithm and $\eta_\ell(K):=\eta(\cT_\ell,K)$
for all $K\in \cT_\ell$. 
Some comments are in order before the axioms of adaptivity are reviewed.

\begin{rem}[exact solve]\label{exact solve}
The main idealisation of this paper is the assumption on  {\em exact solve}  in AMFEM.
An optimal practical algorithm (optimal also with respect to the total run time of the overall algorithm with 
multilevel methods and nested iteration) has to overcome further difficulties 
beyond the scope  of this paper. 
An iterative solver has to be employed in practice and the {\em  termination} of which has to be monitored.  
The computed approximation $\widetilde{\eta}(\cT_\ell)$ to the error estimators $\eta_\ell$
are based on computed approximations $ \widetilde{\Psi_\ell} $ to the discrete solution $ \Psi_\ell $ and the error
$\trinl \Psi_\ell -\widetilde{\Psi}_\ell  \trinr_{\NC} $ has to be
controlled. A practical termination criterion reads 
$\trinl \Psi_\ell -\widetilde{\Psi}_\ell  \trinr_{\NC} \le \kappa  \trinl \widetilde{\Psi_\ell} \trinr_{\NC}$ 
for a {small positive constant} $\kappa$. This could be  guaranteed e.g. by some
Newton-Kantorovich theorem in the finite-dimensional nonlinear discrete problem. A perturbation analysis
enables optimal convergence rates in the general case as in  \cite{CCJGSINUM,Pretetal};
{numerical  experiments can be found in \cite{GMNN_Morley}}.

In the absence of additional information on $\Psi$, the discrete problem may have multiple
solutions and the {\em  selection}  of one in AMFEM is less clear. Moreover,   
$\varepsilon_0, \delta_0>0$ from Theorem~\ref{aprioriestimate}.a exist but are not easy to quantify in general. 
The proposed version of AMFEM has an initial phase with uniform mesh-refining steps monitored with
the input parameter $\delta$. One reason to choose $\delta>0$ small is that $\delta\le\delta_0$ 
resolves the nonlinearity in the sense that it guarantees
the existence of a unique discrete solution near $\boldmath \Psi$. 
\end{rem}

\begin{rem}[input parameter]
The optimal convergence rates follow from 
Theorem~\ref{thmOptimalrates} below under the conditions  $0<\delta,\theta\ll 1$
sufficiently small.
 The  choice of the bulk parameter $\theta<1/(1+\Lambda_1^2\Lambda_3)$ 
 below is independent of $\delta\le \delta_0$ (but depends on $\Psi,\delta_0,\varepsilon_0$).
Several arguments  in the analysis of {\bf (A3)}-{\bf (A4)} below require $\delta>0$ to be very small
(possibly much smaller than $\delta_0$) and it is conjectured  that this is not a technical artefact. 
\end{rem}

\begin{rem}[marking]
Recall the sum convention \eqref{eq:def_estim} for the meaning of 
the bulk criterion \eqref{sAfem:eq:bulkA}. A greedy algorithm for the computation of
a subset $\cM_\ell^*$ with $\theta \: \eta^2_\ell (\cM_\ell^* ) \leq \eta_\ell^2$
and {\em minimal  cardinality}  $|\cM_\ell^*|$
may first sort the triangles in $\cT_\ell$ according to the size of
its estimator contribution $\eta_\ell(K)$. Quick sort may lead to
superlinear computational costs and is circumvented in \cite{Stev07} by 
computing a subset  $\cM_\ell$  of {\em almost minimal 
cardinality} $|\cM_\ell |$ with \eqref{sAfem:eq:bulkA} and 
 $|\cM_\ell|\le C_{\text{am}} |\cM_\ell^*|$  for a universal constant  $C_{\text{am}}\ge 1$.
\end{rem}

\begin{rem}[refine]
{The procedure} $\textsc{Refine}$ specifies 
the newest vertex bisection (NVB)  with completion (to avoid hanging nodes). 
The output $\Tl{\text{out}}:=\textsc{Refine}(\Tl{\text{in}},\M) \in \bT(\cT_\text{in})$  
is {the smallest refinement} of $\Tl {\text{in}}$ with NVB of \cref{fig1} and 
 $M\in \Tl{\text{in}}\setminus \Tl{\text{out}}$.
The initial condition  of  $\cT_{\text{init}}$ carries over to the first triangulation  $\cT_{0}$
because of the uniform refinements with NVB.
More details  may be found in {\cite{BinevDahmenDevore04,KMPDPD2013,Stevenson08}}.
 \end{rem}

\subsection{Axioms of Adaptivity}\label{axioms}
Recall the {\em 2-level notation}: Each triangulation $\T\in\TO(\delta)$ 
(resp. its refinement $\hT\in\TO(\cT)$) 
leads to a unique 
discrete solution $\Psi_{\rm M}=(u_{\rm M},v_{\rm M}) \in \bV(\T)$  
(resp. $\widehat{\Psi}_{\rm M}=(\widehat{u}_{\rm M},\widehat{v}_{\rm M}) \in \bV(\hT)$)
to \eqref{vformdNC} with $\trinl \Psi-\Psi_{\rm M}\trinr_{\rm pw} \le \varepsilon_0$ 
(resp. $\trinl \Psi-\widehat{\Psi}_{\rm M}\trinr_{\rm pw} \le \varepsilon_0$).
This defines  the (global) distance 
\[
\delta(\T,\hT):=  \trinl \widehat{\Psi}_{\rm M} - {\Psi}_{\rm M} \trinr_{\rm pw} 
\]
of $\T\in\TO$ and its refinement $\hT\in\TO(\T)$ as a global non-negative real number. 
Recall the definition \eqref{estimators} of $\eta(\cT,K)$ for all $K\in\cT\in\bT(\delta)$
and specify, for  fixed $\T\in\TO$ and its  fixed
refinement $\hT\in\TO(\T)$,
\[
\eta(K):=\eta(\cT,K)\quad\text{and}\quad
\widehat{\eta}(T):=\eta(\widehat{\cT},T) \quad\text{for }K\in\cT
\text{ and } T\in\widehat{\cT}
\]
and adapt the sum conventions  \eqref{eq:def_estim} for the short-hand notation
$\eta$ and $\widehat{\eta}$ in the axioms {\bf (A1)}-{\bf (A3)}  
with universal constants $\Lstab$, $\Lred$, and $\LdRel$ in 
\begin{axioms}
\item[(A1)]Stability. \hfil
		$\displaystyle \abs{\widehat{\eta}(\T \cap \hT) - \eta(\T \cap \hT)} \leq \Lstab \delta(\T,\hT)$.
	
\item[(A2)] Reduction. 	\hfil
		$\displaystyle \widehat{\eta}(\hT \setminus \T) \leq 
		2^{-1/4}\eta(\T \setminus \hT) + \Lred \delta(\T,\hT)$.
\\[0mm]	
\item[(A3)] Discrete Reliability. \hfil
		$\displaystyle \delta^2(\T, \hT) \leq \LdRel \eta^2(\T \setminus \hT )$.

\item[(A4)] Quasiorthogonality. \hfil			
$\displaystyle \sum_{k=\ell}^{\infty} \delta^2(\Tl k, \Tl {k+1})
			\leq \Lqo \eta_{\ell}^2$ for all $\ell \in \NO$.
\end{axioms}
The notation in the axiom {\bf (A4)}  solely concerns the outcome  $\cT_{\ell}$ and 
$\eta_{\ell}$ of AMFEM  with a universal constant $\Lqo $ and already asserts that the left-hand side
is a converging sum. 

The subsequent section provides the proofs of all those four axioms and then allows
the application of the abstract theorem for optimal convergence rates. Recall the 
definitions of $\TT{\delta}$ for $0<\delta<1$ and $\TT{N}$ for $N\in\bN_0$ in Subsection~\ref{subsecTriangulationsanddiscretespaces} and define the set 
\[
\TT{\cT_0,N}:=\{ \cT\in\TT{\cT_0}: |\cT|\le N+|\cT_0|\} 
\]
of all admissible refinements of $\cT_0$ with at most $N\in\bN_0$ extra triangles.

\begin{thm}[optimal rates in adaptive FEMs \cite{cc14,cc_hella_18}]\label{thmOptimalrates} 
Suppose  {\bf  (A1)}-{\bf  (A4)},
$0<\theta <\theta_0:= 1 /(1+\Lstab^2\LdRel)$ 
in AMFEM with  output  $(\Tl \ell)_{\ell \in \mathbb N_0}$ and 
$(\eta_\ell)_{\ell \in \mathbb N_0}$ and let $s>0$.
Then (there exist equivalence constants in)
\begin{align}\label{eq:optim}
\sup_{\ell \in \NO}\left(1+\abs{\Tl \ell} - \abs{\Tl 0}\right)^{s}  \eta_\ell 
\approx \sup_{N\in \NO} (1+N)^s \min \eta(\TT{\cT_0,N})
\end{align}
with the minimum $\min \eta(\TT{\cT_0,N})$ of all 
$\eta(\cT)$ with $\cT\in \TT{\cT_0,N}$ for $N\in\bN_0$. 
\end{thm}

\begin{proof}
The formulation of this theorem is a simplified version of the results in
\cite{cc14,cc_hella_18} based on the seminal paper \cite{Stev07} for the special case 
$\cT_0\equiv\cT_{\text{\rm init}}$  (leave out the uniform refinement steps in the beginning). 
To enable unique discrete solutions near a regular solution $\Psi$, the present algorithm (AMFEM) involves the computation of $\cT_0$  and then runs a  standard adaptive algorithm. Consequently, the analysis
of the  standard adaptive algorithm in  \cite{cc14,cc_hella_18} applies and requires the  axioms 
 {\bf  (A1)}-{\bf  (A4)}  to hold solely for $\cT\in \TT{\cT_0}$ to {guarantee}  \eqref{eq:optim}. 
 As a consequence,  the   equivalence constants  (behind the notation $\approx$) in \eqref{eq:optim} 
depend on all parameters  $\delta$,  $\theta$,  $\cT_0$, and $s$. 
\end{proof}

The point of this paper is the verification of  {\bf  (A1)}-{\bf  (A4)} for small positive $\delta<1$
to prove the main result of optimal rates.

\begin{thm}[optimal rates in (AMFEM)]\label{thmOptimalratesAMFEM} 
Given a regular solution $\boldmath \Psi$ to  \eqref{vform_cts} and an initial 
triangulation $\cT_{\rm init}$,  there exist positive constants
$\overline{\delta}$, $\overline{\theta}<1$ such that the algorithm (AMFEM)
runs for all $0<\delta\le \overline{\delta}$ and $0<\theta\le \overline{\theta}$
with an output  $(\Tl \ell)_{\ell \in \mathbb N_0}$ and 
$(\eta_\ell)_{\ell \in \mathbb N_0}$ that satisfies  \eqref{eq:optim} for all $s>0$
with equivalence constants  (behind the notation $\approx$), 
which  depend on $\boldmath \Psi$, $\cT_{\rm init}$,  $\overline{\delta}$, $\overline{\theta}$, 
and $s$ but are independent of $\delta$ and $\theta$. 
\end{thm}

\noindent The proof is based on Theorem~\ref{thmOptimalrates} {for small parameters $\delta$ and   $\overline\theta<\theta_0$} and will be completed in Subsection~\ref{ProofofTheoremthmOptimalratesAMFEM}.
 
\begin{rem}[pre-asymptotic range]
The convergence rate is an intrinsically asymptotic concept and does not deteriorate if
$\delta$ or $\theta$ in (AMFEM) are chosen far too small. The computational costs and  the 
overall pre-asymptotic range, however, crucially depend on $\delta$ and may become larger and
larger as $\delta$ approaches zero. In case of a regular solution close to a bifurcation point
(with multiple solutions of small difference) the  restrictions  
$\delta\le\overline{\delta}\le \min\{\delta_0,\delta_1\}$
from  Theorem~\ref{dis_stab}  and \ref{aprioriestimate} may already enforce  $\overline{\delta}$ 
to be very small. 
\end{rem}

\begin{rem}[nonlinear approximation]\label{rem:nonlinearapproximation}
 The equivalence \eqref{eq:optim} asserts optimal
 convergence rates (for $s>0$ is arbitrary) in terms of the error estimators.
 The efficiency in  Theorem~\ref{reliability_NC} transforms this to rate optimality
 with respect to nonlinear approximation classes  \cite{BinevDahmenDevore04}  of the total error
$  \trinl\Psi-\Psi_{\rm M}\trinr_{\rm pw} +{\rm osc}_0(f,\cT) $.
\end{rem}

\section{Proofs}
This section verifies {\bf  (A1)}-{\bf  (A4)}  and 
Theorem~\ref{thmOptimalratesAMFEM}.  
Throughout this section, $0<\delta\le \delta_0<0$ with  $\delta_0,\varepsilon_0 >0$ from 
Theorem~\ref{aprioriestimate} and  the 2-level notation of (the beginning of)  
Subsection~\ref{axioms} applies to
$\cT\in \bT(\delta)$, $\hT \in \bT(\cT)$,  $\Psi_{\rm M}=(u_{\rm M},v_{\rm M}) \in \bV(\T)$
with $\trinl \Psi-\Psi_{\rm M}\trinr_{\rm pw} \le \varepsilon_0$, 
$\widehat{\Psi}_{\rm M}=(\widehat{u}_{\rm M},\widehat{v}_{\rm M}) \in \bV(\hT)$
with $\trinl \Psi-\widehat{\Psi}_{\rm M}\trinr_{\rm pw} \le \varepsilon_0$, 
$\eta:=\eta(\cT,\bullet)$, and $\widehat\eta:=\eta(\hT,\bullet)$;
whereas  $\cT_\ell$, $\Psi_\ell$,   and 
$\eta_\ell:=\eta(\cT_\ell)$
refer to  the output of AMFEM. 

\subsection{Proof of  stability  (A1)}
The proofs of {\bf  (A1)} and {\bf (A2)} rely on triangle and Cauchy inequalities plus one lemma. 

\begin{lem}[discrete jump control {\cite[Lem. 5.2]{cc_hella_18}}]\label{djc} There exists a universal constant $C_{\rm jc}$, which depends on the shape regularity in $\TO$ and the degree $k \in {\mathbb N}_0$, such that any $\T \in \TO$ and $g \in P_k(\T)$ with its jumps 
$$
[g]_E=
\begin{cases}
(g|_{T_+})|_E - (g|_{T_-})|_E  \; \mbox {for } E \in \E(\Omega) \mbox{ with } E=\partial T_+ \cap \partial T_{-},\\
g|_E  \mbox { for } E \in \E(\partial \Omega) \cap \E(K) \\
\end{cases}
 $$
across any side $E \in {\E}$ (i.e., with respect to $\T \in \TO$)  satisfy 
$$ 
\mbox{}\hspace{3cm}
\sum_{K \in \T} |K|^{1/2} \sum_{E \in \E(K)} \|[g]_E\|^2_{L^2(E)} \le C_{\rm jc}^2\|g\|^2_{2}. 
\hspace{2cm}\qed
$$
\end{lem}

\begin{thm}[Stability {\bf (A1)}]\label{a1} 
{\bf (A1)} holds for all $\cT\in\bT(\delta_0)$ and all $\widehat{\cT}\in\bT(\cT)$. 
\end{thm}

\begin{proof}
The proof follows \cite{cc14,CKNS08,cc_hella_18}  for linear second-order problems 
with focus on  the nonlinear contributions.
The definitions of $\widehat{\eta}(\T \cap \hT) $ and $\eta(\T \cap \hT)$ 
in Subsection~\ref{axioms}
and a first reverse triangle inequality in  ${\mathbb R}^m$ with the number $m:= |{\T \cap \hT}|$  of triangles in ${\T \cap \hT}$ lead to
\[
 |{\widehat{\eta}(\T \cap \hT) - \eta(\T \cap \hT)}|^2 \le 
\text{$ \sum $}_{K \in \T \cap \hT} 
 \left( \widehat{\eta}(K)- \eta(K)\right)^2.
\]
For  $K \in \T \cap \widehat{T}$,
each of the terms $\widehat{\eta}(K)$ and $\eta(K)$ allows  
a second and third reverse triangle inequality in ${\mathbb R}^8$ and  
$L^2(K)$ or $L^2(E)$ for $E\in\cE(K)$.
This  and  {$|[D^2({\widehat \varphi}_{\M} -{\varphi}_{\M}) ]_E \tau_E| \le
 |[D^2({\widehat \varphi}_{\M} -{\varphi}_{\M}) ]_E |$  for $\varphi=u,v$} with the Frobenius matrix 
 norm $|\bullet|$ in ${\mathbb R}^{2 \times 2}$ 
 result in 
\begin{align} \label{estimator123}
& \big( \widehat{\eta}(K)- \eta(K)\big)^2  
\le |K|^2 \|[\widehat{u}_{\rm M}, \widehat{v}_{\rm M}] - [u_{\M},v_{\M}]\|^2_{2,K} + |K|^2 \|[\widehat{u}_{\rm M}, \widehat{u}_{\rm M}] - [u_{\M},u_{\M}]\|^2_{2,K} \nonumber \\
& \; + |K|^{1/2} \sum_{E \in \E(K)} 
\bigg(\|[D^2({\widehat u}_{\rm M} -{u}_{\rm M}) ]_E  \|^2_{L^2(E)} +
 \|[D^2({\widehat v}_{\rm M} -{v}_{\rm M}) ]_E \|^2_{L^2(E)}\bigg).
\end{align}
Triangle, Cauchy-Schwarz inequalities, and an inverse estimate (here indeed an equality
for $\|D^2 v_\M\|_{2,K}$ is $|K|^{1/2}$ times the Frobenius norm of the constant Hessian 
$D^2 v_\M|_K $) 
show 
\begin{align*} 
  &\|[\widehat{u}_{\M}, \widehat{v}_{\M}] - [u_{\M},v_{\M}] \|_{2,K} \le \| [\widehat{u}_{\M} - u_{\M},\widehat{v}_{\M} ]\|_{2,K} + 
\|[u_\M, \widehat{v}_{\M} - v_{\M}]\|_{2,K}  \nonumber \\
& \le \|   D^2(\widehat{u}_{\M} - u_{\M})\|_{2,K} \| D^2 \widehat{v}_{\M}\|_{\infty,K} + 
\| D^2 {u}_{\M}\|_{\infty,K} \| D^2(\widehat{v}_{\M} - v_{\M})\|_{2,K}  \nonumber \\
& = |K|^{-1/2} (  \|D^2(\widehat{u}_{\M} - u_{\M})\|_{2,K} \| D^2 \widehat{v}_{\M}\|_{2,K} + 
\| D^2 {u}_{\M}\|_{2,K} \|D^2(\widehat{v}_{\M} - v_{\M})\|_{2,K})
\nonumber \\
& \le  |K|^{-1/2}  ( \|D^2 {u}_{\M} \|^2_{2,K} + \|D^2 \widehat{v}_{\M} \|^2_{2,K} )^{1/2} 
 \|D^2(\widehat{\Psi}_{\M} - \Psi_{\M})\|_{2,K}.
\end{align*}
This proves an estimate for the first  term on the right-hand side of \eqref{estimator123},
\begin{align} 
|K|^2 \|[\widehat{u}_{\M}, \widehat{v}_{\M}] - [u_{\M},v_{\M}] \|^2_{2,K} \le |K| (\| D^2_\NC \Psi_\M\|^2_{2,K} + 
\|D^2_\NC \widehat{\Psi}_\M \|^2_{2,K}) \|D_\NC^2(\widehat{\Psi}_{\M} - \Psi_{\M})\|^2_{2,K}. \nonumber
\end{align}
The substitution of $v_\M$ (resp. $\widehat{v}_\M$) by $u_\M$  (resp. $\widehat{u}_\M$) 
provides an analog inequality.  
The sum of those two estimates and the sum over all $K \in \T \cap \hT$ 
with $ |K| \le h_{\rm max}^2\le  |\Omega|$  show 
\begin{align*} 
& \sum_{K \in \T \cap \hT} |K|^2 \left(\|[\widehat{u}_{\M}, \widehat{v}_{\M}] - [u_{\M},v_{\M}] \|^2_{2,K}  + \|[\widehat{u}_{\M}, \widehat{u}_{\M}] - [u_{\M},u_{\M}] \|^2_{2,K} \right)
\nonumber \\
& \quad \le \sum_{K \in \T \cap \hT} 2 |K| (\| D^2_\NC \Psi_\M\|^2_{2,K} + 
\|D^2_\NC \widehat{\Psi}_\M \|^2_{2,K}) 
\|D_\NC^2(\widehat{\Psi}_{\M} - \Psi_{\M})\|^2_{2,K} \nonumber \\
& \quad \le 2 h^2_{\rm max} (\trinl \Psi_{\M}\trinr_{\NC}^2 + \trinl \widehat{\Psi}_{\M}\trinr_{\NC}^2 )
\trinl \widehat{\Psi}_{\M} - \Psi_{\M} \trinr^2_{\NC} 
\le 4 h^2_{\rm max}  M^2 \delta^2(\T,\hT)
\end{align*}
with the abbreviation  $M:= \trinl {\Psi}  \trinr+\varepsilon_0 \lesssim \|f\|_{-1}$. 

The analysis of the jump terms in  \eqref{estimator123} 
is the same as in \cite{CCDGJH14,DG_Morley_Eigen}. 
With the substitution of $\cT$ by $\hT$, Lemma~\ref{djc} applies (componentwise)   to the jump contributions 
 $D_{\NC}^2(\widehat{\Psi}_\M- \Psi_\M) \in P_0(\widehat{\T}; {\mathbb R}^{2\times 2})$ 
 in the sum of \eqref{estimator123} over all $K\in\cT\cap\hT$.
This proves {\bf (A1)} with   $\Lambda^2_1:= C_{\rm jc}^2   + 4M^2 |\Omega|$.  
\end{proof}

\begin{rem}[volume terms]\label{remvolumeterms0fA1}   
Subsection~\ref{subsectprelinimariesProofofquasiorthogonality(A4)} revisits 
the above proof for the volume terms $\mu^2(K)$   in $\eta^2(K)$,  
 \[
\mu^2(K):=  |K|^2\left( \|[u_\M,v_\M]+f \|_{2,K}^2+\|[u_\M,u_\M]\|_{2,K}^2\right) 
\quad\text{for all }K\in\cT.
\]
The formula defines  the volume contributions  $\widehat\mu^2(T)$ in  $\widehat\eta^2(T)$ 
with the substitution of  $u_\M,v_\M,K$  by  $\widehat{u}_\M,  \widehat{v}_\M, T$  for  $T\in\hT$.
The proof of {\bf (A1)}  shows   the refined estimate
\begin{equation}\label{eqnewccA1forproofofA4}
| \widehat\mu (\cT\cap\hT)-\mu(\cT\cap\hT)|
\le 2h_{\max}M\delta(\cT,\hT)
\end{equation} 
with an adaptation of the sum convention \eqref{eq:def_estim} to define $\mu(\cT\cap\hT)$
resp. $\widehat\mu(\cT\cap\hT)$.
\end{rem}

\subsection{Proof of reduction (A2)}
The triangle $T\in\hT\setminus\cT$ is included in exactly one $K\in\cT$ in the NVB
refinement and $T\subsetneqq K$ proves $|T|\le |K|/2$ to generate the reduction factor
$2^{-1/4}$ displayed in {\bf (A2)}.

\begin{thm}[Reduction {\bf (A2)}] {\bf (A2)} 
holds for all $\cT\in\bT(\delta_0)$ and all $\widehat{\cT}\in\bT(\cT)$.
\end{thm}

\begin{proof}
Given any triangle $K \in \T \setminus \hT$, the square of the error estimator for 
the $m\ge 2$ finer triangles $T \in \hT (K):= \{T \in \hT: T \subset K\}$  reads
\begin{align*}
 \widehat{\eta}^2(\hT (K) ) & =\sum_{T \in  \hT (K) }   \bigg( |T|^2\left(\left\|[\widehat{u}_\M,
 \widehat{v}_\M]+f\right\|_{2,T}^2+\left\|[\widehat{u}_\M,\widehat{u}_\M]\right\|_{2,T}^2\right) \nonumber \\
& \; \;+ |T|^{1/2} \sum_{F \in {\cal E}(T) }\left(\|D^2 \widehat{u}_\M]_F\tau_F\|_{L^2(F)}^2+\|[D^2 \widehat{v}_\M]_F\tau_F\|_{L^2(F)}^2\right)\bigg).
\end{align*}
Various triangle inequalities  (in Lebesgue and Euclid norms) show %
$\widehat{\eta}(\hT (K) )\le S_1+S_2$ for 
\begin{align*} 
 S_1^2&:=   \sum_{T \in \hT(K)} 
\bigg( |T|^2\left(\left\|[{u}_\M,
 {v}_\M]+f\right\|_{2,T}^2+\left\|[{u}_\M,{u}_\M]\right\|_{2,T}^2\right) \nonumber \\
& \qquad \; + |T|^{1/2} \sum_{F \in {\cal E}(T) }\left(\|[D^2 {u}_\M]_F\tau_F\|_{L^2(F)}^2+\|[D^2 {v}_\M]_F\tau_F\|_{L^2(F)}^2\right) \bigg) 
\le 2^{-1/2} \eta^2(K).
 \end{align*}
The proof of this utilises $|T|^{1/2}\le 2^{-1/2} |K|^{1/2}$ and a careful rearrangement of the jumps 
(that vanish over edges $E\in\cE(T)$ inside $K$ and sum up to the $L^2$
contribution along $\partial K$) and the volume contribution. The second term 
\begin{align*} 
S_2^2 & := \sum_{T \in \hT(K)}  \bigg( |T|^2 (\|[\widehat{u}_{\M}, \widehat{v}_{\M}] - [u_{\M},v_{\M}]\|^2_{2,T} + \|[\widehat{u}_{\M}, \widehat{u}_{\M}] - [u_{\M},u_{\M}]\|^2_{2,T}) \nonumber \\
&+ |T|^{1/2} \sum_{F \in \E(T)} \bigg(\|[D^2({\widehat u}_{\M} -{u}_{\M}) ]_F \tau_F \|^2_{L^2(F)} +
 \|[D^2({\widehat v}_{\M} -{v}_{\M}) ]_F \tau_F \|^2_{L^2(F)} \bigg) 
 \end{align*}
is analysed as in the previous subsection. 
The arguments eventually prove 
\begin{align*}
 \sum_{T \in \hT(K)}  |T|^2 \|[\widehat{u}_{\M}, \widehat{v}_{\M}] - [u_{\M},v_{\M}]\|^2_{2,T}
\le |K| M^2  \|D_\NC^2(\widehat{\Psi}_{\M} - \Psi_{\M})\|^2_{2,K}
\end{align*} 
and the analog estimate with 
$v_\M$ (resp. $\widehat{v}_\M$) substituted by $u_\M$  (resp. $\widehat{u}_\M$). 
The analysis of the jump terms 
is the same as in \cite{CCDGJH14,DG_Morley_Eigen}
and  Lemma~\ref{djc} (applied to $\hT$ rather than $\cT$)
eventually leads to  {\bf (A2)} with  
 $\Lambda^2_2:= C_{\rm jc}^2 + 2M^2 |\Omega|$. 
\end{proof}

\begin{rem}[assumptions]   
The restriction to $\cT\in\bT(\delta_0)$ in {\bf (A1)} -{\bf (A2)} guarantees the definition of the error estimators
via the discrete solution through Theorem~\ref{aprioriestimate}. 
This is exclusively  for notational consistency:  {\bf (A1)} -{\bf (A2)}
hold for any $\Psi_\M\in V(\cT)$ and $\widehat{\boldmath \Psi}_\M\in V(\hT) $ and solely 
$\Lambda_1,\Lambda_2$ depend on a universal upper bound $2M$ for 
$\trinl   {\boldmath\Psi}_\M  \trinr_{\NC} + \trinl   \widehat{\boldmath \Psi}_\M  \trinr_{\NC}$.
\end{rem}

\begin{rem}[volume terms]\label{remvolumeterms0fA2}   
Subsection~\ref{subsectprelinimariesProofofquasiorthogonality(A4)} revisits 
the above arguments solely for the volume terms $\mu^2(K)$   in $\eta^2(K)$  
for $K\in\cT$ (resp.  $\widehat\mu^2(T)$ in  $\widehat\eta^2(T)$  for $T\in\hT$)
from Remark~\ref{remvolumeterms0fA1}. 
With an adaptation of the sum convention \eqref{eq:def_estim} for $\mu$ and $\widehat\mu$, 
the proof of {\bf (A2)} shows  
\begin{equation}\label{eqnewccA2forproofofA4}
\widehat\mu (\hT\setminus \T)\le 2^{-1/2}\mu(\cT\setminus\hT)
+  2^{1/2} h_{\max}M\delta(\cT,\hT).
\end{equation} 
\end{rem}

\subsection{Proof of discrete reliability (A3)}
The parameters $\delta_3$ and $\Lambda_3$ in the following version of {\bf (A3)} depend on 
the regular solution $\Psi$ and its regularity in Theorem~\ref{ap}, on  $\delta_0,\varepsilon_0$ 
(resp. $\delta_1,\beta_1$)
from Theorem~\ref{aprioriestimate} (resp. Theorem~\ref{dis_stab}), on $\T_\text{init}$ and $\Omega$ 
with the regularity index $\gamma$. 

\begin{thm}[discrete reliability {\bf (A3)}]\label{thma3}
There exists positive  $\delta_3\le \min\{\delta_0,\delta_1\}$ and $\Lambda_3$ 
such that
$\delta^2(\T, \hT) \leq \LdRel \eta^2( \T \setminus \hT )$ holds 
for any $\cT\in \bT(\delta_3)$ with refinement  $ \hT \in \TT{\T}$.
\end{thm}

\begin{proof}
Given any  refinement  $ \hT \in \TT{\T}$ of $\T\in \bT$, 
the interpolation operator $I_{\rm M} $ 
of Lemma \ref{Morley_Interpolation} maps  $M(\hT)\to M(\cT)$. The converse 
operation in \cite{CCDGJH14,DG_Morley_Eigen} relies on a 
discrete Helmholtz decomposition.  
This paper follows   \cite{CCP} with 
a right-inverse  $\widehat{I}_{\rm M} E_{\rm M}$.
The key idea is  first to compute the companion operator $E_{\rm M} v_{\rm M}$ for some $v_\M\in \M(\cT)$
and second to apply the interpolation operator $\widehat{I}_{\rm M} $ of 
Lemma \ref{Morley_Interpolation} on the finer triangulation $\hT$ (rather than $\cT$).
This leads to $\widehat{I}_{\rm M} E_{\rm M}: {\rm M}(\T) \rightarrow {\rm M}(\hT)$ with
${I}_{\rm M}(\widehat{I}_{\rm M} E_{\rm M})=1$ in $\M(\cT)$. 

A modification of this idea is performed in   \cite[Def 6.9]{CCP}, \cite{DG_Morley_Eigen}
to define  an operator 
$J_2:P_2(K) \rightarrow HCT(K) +P_5(K)$ for each $K\in\cT$ such that 
 $\Psi_{\M}^*:=\widehat{I}_{\M}(J_2 \Psi_\M) \in \bV(\hT)$  
 is well defined \cite[Lem. 6.14]{CCP} and satisfies   \cite[Thm. 6.19]{CCP} that
\begin{align}\label{t32}
C_1^{-1}\trinl   {\Psi}^*_\M-\Psi_\M  \trinr_{\NC} \le 
\bigg( \sum_{E \in {\mathcal E} \setminus \widehat{{\mathcal E}}}
|\omega_E|^{1/2}_E  \|[D^2 \Psi_{\M}]_E  \tau_{E} \|^2_{L^2(E)}\bigg)^{1/2}  \le
\eta(\T \setminus \hT )
\end{align}
with the mesh-size factor $|\omega_E|^{1/2}\approx \text{\rm diam}(E)$ for any
edge $E$ with its edge-patch $\omega_E$ of area $|\omega_E|$
and  some universal  constant $C_1$ (that depends solely on $\cT_{\text{\rm init}}$).
This restricts the sum over all edges $E$ in  \eqref{t32} to those, which are coarse but not fine. 
The estimate  \eqref{t32}  and a triangle inequality imply
\begin{align}\label{triangle}
 \delta (\T,\widehat{\T}) & \le \trinl \widehat{\Psi}_\M - \Psi_{\M}^* \trinr_{\NC} 
 + C_1\eta(\T \setminus \hT ).
\end{align}
It  remains to control 
$\trinl\widehat{\Psi}_\M - \Psi_{\M}^* \trinr_{\NC} $ for  the Morley function 
$\widehat{\Psi}_\M - \Psi_{\M}^*\in \bV(\hT)$.
The discrete stability in Theorem~\ref{dis_stab} leads to  some 
${\widehat{\bf y}}_{\M} \in { {\mathcal V}}(\hT)$ with  
$\trinl {\widehat{\bf y}}_\M \trinr_{\NC} \le 1/ {\beta_1}$ and
\begin{align} \label{t}
\trinl \widehat{\Psi}_\M - \Psi_{\M}^* \trinr_{\NC}  &  = DN_h(\Psi, \widehat{\Psi}_\M- \Psi_{\M}^* , \widehat{\bf y}_\M). 
\end{align}
The boundedness in   
Lemma~\ref{b_dG}.a and $C_2:=(1+ 2\sqrt{2}C_{\text{\rm de}}\trinl \Psi  \trinr )/ {\beta_1} $ show
 \[
 DN_h(\Psi, \Psi_\M-\Psi_{\M}^*,\widehat{\bf y}_\M)\le C_2 \trinl   \Psi_\M-\Psi_{\M}^* \trinr_{\NC}
 \le C_1C_2 \eta(\T \setminus \hT )
 \]
with  \eqref{t32} in the last step.
The combination of this with \eqref{triangle}-\eqref{t} proves
\begin{align} \label{t2a1}
\delta (\T,\widehat{\T}) \le DN_h(\Psi, \widehat{\Psi}_\M-\Psi_\M,  \widehat{\bf y}_\M)
 + C_1(1+C_2) \eta(\T \setminus \hT ).
\end{align}
Recall that   $\Psi_{\rm M}\in\bV(\cT)$ (resp. 
$\widehat{\Psi}_{\rm M}\in\bV(\hT)$) solves the discrete problem 
with respect to $\T$ (resp.  $\hT$). 
Lemma \ref{Morley_Interpolation}.a  shows
$A_\NC(\Psi_\M,\widehat{\bf y}_\M ) = A_\NC(\Psi_\M,I_\M\widehat{\bf y}_\M )$, {for $I_\M: {\cal V}(\widehat{\cT} ) \rightarrow {\cal V}({\cT} ) $ defined in Section \ref{sec:int}}.
This and elementary algebra 
(with the symmetry of $B_{\NC}(\bullet,\bullet,\bullet)$ in the first two variables) lead to 
\begin{align} \label{t2a}
DN_h(\Psi, \widehat{\Psi}_{\text{\rm M}}-\Psi_{\text{\rm M}},  
\widehat{\bf y}_{\text{\rm M}})& 
=  B_{\NC}(2\Psi-\widehat{\Psi}_{\text{\rm M}}-\Psi_{\text{\rm M}}, 
\widehat{\Psi}_{\text{\rm M}}-\Psi_{\text{\rm M}}, \widehat{\bf y}_{\text{\rm M}})
\nonumber \\
&\hspace{-20mm}
+F(\widehat{\bf y}_{\text{\rm M}} -  I_{\text{\rm M}} \widehat{\bf y}_{\text{\rm M}}) 
- B_{\NC}({\Psi}_{\text{\rm M}}, {\Psi}_{\text{\rm M}}, \widehat{\bf y}_{\text{\rm M}} 
-  I_{\text{\rm M}} \widehat{\bf y}_{\text{\rm M}})   .
\end{align}
The a~priori error estimate 
$ \trinl \Psi - \widehat{\Psi}_{\text{\rm M}} \trinr_{\NC}$, 
$  \trinl \Psi -{\Psi}_{\text{\rm M}} \trinr_{\NC} \le C(\gamma,\Psi) h^\gamma_{\max} $
from  Theorem~\ref{aprioriestimate} in terms of 
the maximal mesh-size $h_{\max}$ 
and Lemma~\ref{b_dG}.a
result in  
\[
 B_{\NC}(2\Psi-\widehat{\Psi}_{\text{\rm M}}-\Psi_{\text{\rm M}}, 
\widehat{\Psi}_{\text{\rm M}}-\Psi_{\text{\rm M}}, \widehat{\bf y}_{\text{\rm M}})
\le  \sqrt{8}  \beta_1^{-1} C_{\text{\rm dea}} C(\gamma,\Psi) 
 h^\gamma_{\max}\,   \delta (\T,\widehat{\T}).
\]
The last two  contributions in \eqref{t2a} are volume residuals  with the test
function  $\widehat{\bf y}_\M -  I_\M \widehat{\bf y}_\M$, which vanishes
a.e. in each $K\in \T \cap \hT$ (see Remark \ref{rem:new}). This and Lemma~\ref{Morley_Interpolation}.b 
imply
\begin{align*}
 & F(\widehat{\bf y}_\M -  I_\M \widehat{\bf y}_\M) - B_{\NC}({\Psi}_\M, {\Psi}_\M, 
\widehat{\bf y}_\M- I_\M\widehat{\bf y}_\M)   \lesssim
{{\beta_1}}\, 
 \trinl  \widehat{\bf y}_\M -  I_\M \widehat{\bf y}_\M \trinr_{\NC}  \nonumber \\
   & \quad \times \big(\sum_{K\in\cT \setminus \hT} |K|^2(\|f+[{u}_{\M},{v}_{\M}]\|_{2,K}^2 + 
  \|[u_{\M},u_{\M}]\|_{2,K}^2)\big)^{1/2} 
   \le \eta(\T \setminus \hT )
\end{align*} 
with 
$\trinl \widehat{\bf y}_\M -  I_\M \widehat{\bf y}_\M \trinr_{\NC} \le
\trinl\widehat{\bf y}_\M\trinr_{\NC} \le 1/{\beta_1} $ in the last step.
These estimates control the right-hand side in \eqref{t2a}. The resulting estimate and  \eqref{t2a1}
lead to  $C_3\approx 1$ with  
\begin{align} \label{t2a1234}
\left(1- \sqrt{8}C_{\text{\rm dea}}  \beta_1^{-1} C(\gamma,\Psi) h^\gamma_{\max} \right)
\delta (\T,\widehat{\T}) \le C_3   \eta(\T \setminus \hT ).
\end{align}
The estimate \eqref{t2a1234} holds for all triangulations in $\bT(\min\{\delta_0,\delta_1\})$
and the  particular choice 
$\delta_3=\min\{\delta_0,\delta_1$,  $ (   {2 \sqrt{2} C_{\text{\rm dea}}}
C(\gamma,\Psi)/{\beta_1} )^{-1/\gamma}\} $ proves {\bf (A3)} with 
 $\Lambda_3:=4C_3^2$.
\end{proof}

The discrete reliability {\bf (A3)} implies reliability of the error estimators. 

\begin{cor}[reliability]\label{correliability}
Given the exact solution $\Psi$ and $\delta_3$ from Theorem~\ref{thma3},
the discrete solution $\Psi_{\rm M} \in V(\cT)$ for $\cT\in\bT(\delta_3) $ satisfies  
$
\trinl\Psi-\Psi_{\rm M}\trinr_{\rm pw}^2\leq \Lambda_3 \eta^2(\cT)$ .
\end{cor}

\begin{proof}
Given $\cT^{(0)}:=\cT  \in\bT(\delta_3)$, define a sequence of uniform refinements 
by  $\cT^{(k+1)}=   \textsc{Refine}(\cT^{(k)})$ for  any $k\in\bN$.
Let $\hT:=\cT^{(k)}$ for the parameter $k\in\bN$ and notice that the maximal mesh-size in $\hT$
tends to zero as $k\to\infty$. Hence Theorem~\ref{aprioriestimate} guarantees convergence of
$\trinl\Psi- \widehat{\boldmath \Psi}_\M \trinr_{\rm pw}\to 0 $ as  $k\to\infty$. On the other hand,
Theorem~\ref{thma3} shows 
$\trinl \Psi_\M- \widehat{\boldmath \Psi}_\M  \trinr_{\rm pw}^2\le \Lambda_3\eta^2(\cT)$.
Since the upper bound does not
depend on $k\in\bN$, this and a triangle inequality shows 
the assertion in the limit as  $k\to\infty$.
\end{proof}

\subsection{Preliminaries to the proof of quasiorthogonality}\label{subsectprelinimariesProofofquasiorthogonality(A4)}
The quasiorthogonality is always subtle  for nonconforming  schemes and 
requires a careful analysis of the quadratic nonlinear contributions.
The proof departs with two preliminary lemmas 
formulated in the (2-level) notation of {\bf (A1)}-{\bf (A3)}. Recall the notation
$\mu$ resp. $\widehat\mu$  in 
\eqref{eqnewccA1forproofofA4}-\eqref{eqnewccA2forproofofA4} for the volume contributions
of the error estimator $\eta$ resp. $\widehat\eta$ 
and adapt the sum convention \eqref{eq:def_estim} 
with $\mu^2(\T)=\sum_{K\in\cT}\mu^2(K)$ etc.
%
Recall that ${h}_\T \in {P}_0(\cT)$ is the mesh-size in $\cT$ with 
$h_{\max}:=\max h_\cT\le\delta_0$ with $\delta_0$ and $\varepsilon_0$ from 
Theorem~\ref{aprioriestimate}. Suppose $\cT\in \bT(\delta_0)$ and $\hT\in\bT(\cT)$
throughout this section. The following estimate of $\mu$  and $\widehat{\mu}$ in particular implies for any unform refinement $\widehat{\T}$ of $\T$ the reduction $\widehat{\mu}({\widehat{\T}}) \le \frac{3}{4} \mu (\T)$ plus small terms $h_{\rm max}$ times $\delta(\T, \widehat{\T})$ indicated in Subsection 1.3.
 
\begin{lem} \label{claim1} 
The bound $M:= \trinl { \Psi}  \trinr+\varepsilon_0$ satisfies 
\[
\mu^2(\cT\setminus\hT) 
 \le 4\mu^2(\cT) - 4\widehat\mu^2 (\hT) 
+8 h_{\max}M\delta(\cT,\hT)\left(
\widehat\mu (\hT)+\mu(\cT)\right) 
+  24 h^2_{\max}M^2\delta^2(\cT,\hT).
\]
\end{lem}

\begin{proof}
Recall \eqref{eqnewccA2forproofofA4} and  deduce
\[
\widehat\mu^2 (\hT\setminus\T)\le 3/4\, \mu^2(\cT\setminus\hT)
+  6 h^2_{\max}M^2\delta^2(\cT,\hT).
\]
This is equivalent to
\[
2^{-2} \mu^2(\cT\setminus\hT) +\widehat\mu^2 (\hT) -\mu^2(\cT) \le  
\widehat\mu^2 (\T\cap\hT)-\mu^2(\cT\cap\hT)
+  6 h^2_{\max}M^2\delta^2(\cT,\hT).
\]
Recall \eqref{eqnewccA1forproofofA4} and the binomial formula to derive 
\begin{align*}
\widehat\mu^2 (\T\cap\hT)-\mu^2(\cT\cap\hT) \le 
 2h_{\max}M\delta(\cT,\hT)\left(
\widehat\mu (\T\cap\hT)+\mu(\cT\cap\hT)\right). 
\end{align*}
The combination of the previous two displayed estimates concludes the proof.
\end{proof}
\noindent {Recall the definition of $\text{osc}_m(f,\T)$ from Section \ref{subsecTriangulationsanddiscretespaces} and write $\text{osc}_m(f,{\hT})$ for the oscillations with respect to the finer triangulation $\hT$.}
\begin{lem} \label{qo1} 
There exists a constant $C_{\rm qo}$ (depending on $\Psi$, $f$, the 
constants in Theorem~\ref{aprioriestimate},  and $\cT_{\text{\rm init}}$)
such that
\begin{align*}
& A_{\rm{ pw}}(\Psi-\widehat{\Psi}_{\rm M}, \Psi_{\rm M} - \widehat{\Psi}_{\rm M}) \le 
C_{\rm qo}^{1/2} {\big( \trinl \Psi- \widehat{\Psi}_{\rm M} \trinr^2_{\rm pw} 
+  {\rm osc}_m^2(f,\T) \big)^{1/2} }
 \big(   \mu(\T \setminus \hT)+ h_{\max}^{\gamma}
\delta(\cT,\hT)   \big).
\end{align*}
 \end{lem}
 
\begin{proof}
Recall  $D^2_\NC I_{\rm M}=\Pi_0D^2_\NC$ 
from Lemma \ref{Morley_Interpolation}.a,  set
$\widehat{\Phi}_{\text{\rm M}}:= 
\widehat{I}_{\text{\rm M}}(\Psi-\widehat{\Psi}_{\text{\rm M}})$,
and evaluate the discrete equations on the coarse (resp. fine) 
level to derive 
\begin{align*} 
& A_{\NC}(\Psi-\widehat{\Psi}_{{\text{\rm M}}}, \Psi_{{\text{\rm M}}} 
- \widehat{\Psi}_{{\text{\rm M}}}) 
= A_{\NC} (\Psi_{\text{\rm M}}, \Psi-\widehat{\Psi}_{{\text{\rm M}}})  
-A_{\NC}( \widehat{\Psi}_{{\text{\rm M}}},\Psi-\widehat{\Psi}_{{\text{\rm M}}})  \nonumber \\
&\quad
= A_{\NC}(\Psi_{{\text{\rm M}}}, I_{{\text{\rm M}}}(\Psi-\widehat{\Psi}_{{\text{\rm M}}})) 
- A_{\NC}(\widehat{\Psi}_{{\text{\rm M}}}, \widehat{\Phi}_{\text{\rm M}}) \nonumber \\
&  \quad =  F((I_{{\text{\rm M}}}-\widehat{I}_{{\text{\rm M}}})(\Psi-\widehat{\Psi}_{\text{\rm M}})  )- 
 B_{\NC}(\Psi_{{\text{\rm M}}},\Psi_{\text{\rm M}}, ({I}_{{\text{\rm M}}}-\widehat{I}_{\text{\rm M}})(\Psi-\widehat{\Psi}_{\text{\rm M}})) \nonumber \\
&  \quad \quad + 
B_{\NC}(\widehat{\Psi}_{\text{\rm M}}, \widehat{\Psi}_{\text{\rm M}}, 
\widehat{\Phi}_{\text{\rm M}})
- B_{\NC}(\Psi_{{\text{\rm M}}},\Psi_{\text{\rm M}}, 
\widehat{\Phi}_{\text{\rm M}}) =:S_3+S_4.
\end{align*}
The definitions of $F_{\NC}(\bullet)$, $B_{\NC}(\bullet,\bullet,\bullet)$, the Cauchy
 inequality,  and  
 $(1-I_{\text{\rm M}})\widehat{\Phi}_{\text{\rm M}}=0$ a.e. in  $K \in T \cap \hT$ 
 prove  (with the vector 
 $(f +  [u_{\text{\rm M}}, v_{\text{\rm M}}], -\frac{1}{2} [u_{\text{\rm M}},u_{\text{\rm M}}])\in L^2(\Omega;\bR^2)$
  and the  scalar product $\cdot$ in $\bR^2$) that
\begin{align*} S_3&:=
F((I_{{\text{\rm M}}}-\widehat{I}_{{\text{\rm M}}})(\Psi-\widehat{\Psi}_{\text{\rm M}})  ) -
 B_{\NC}(\Psi_{{\text{\rm M}}},\Psi_{\text{\rm M}}, ({I}_{{\text{\rm M}}}- \widehat{I}_{\text{\rm M}})(\Psi-\widehat{\Psi}_{\text{\rm M}})) \nonumber \\
 &= \sum_{K \in \T\setminus\hT} \int_K 
 (f +  [u_{\text{\rm M}}, v_{\text{\rm M}}], -\frac{1}{2} [u_{\text{\rm M}},u_{\text{\rm M}}])\cdot (I_{{\text{\rm M}}}-\widehat{I}_{{\text{\rm M}}})(\Psi-\widehat{\Psi}_{\text{\rm M}}) \; \dx  \\
&\le  \mu(\T \setminus \hT)
 \big(\sum_{K \in \T \setminus \hT} h_K^{-4} 
  \| (\widehat{I}_{{\text{\rm M}}}-I_{{\text{\rm M}}}) (\Psi- \widehat{\Psi}_{{\text{\rm M}}})\|^2_{2,K} \big)^{1/2}  \lesssim \mu(\T \setminus \hT)
   \trinl \Psi - \widehat{\Psi}_{{\text{\rm M}}} \trinr_{\NC}.
\end{align*}
{The last estimate follows from rewriting $\widehat{I}_{{\text{\rm M}}}-I_{\text{\rm M}} =({\widehat I}_{\rm M} -1)+ (1-I_{\rm M}) $, a triangle inequality  to separate out the terms $(\widehat{I}_\M-1)(\Psi-\widehat{\Psi}_\M)$ and $(1-{I_\M})(\Psi-\widehat{\Psi}_\M)$ and then Lemma ~\ref{Morley_Interpolation}.b and the Pythogoras theorem (resp. Lemma ~\ref{Morley_Interpolation}.b)  establish the stability of the resulting first  (resp. second) term}.
The triangle inequality and   Lemma~\ref{b_dG}.a-b (with $\hT$ replacing $\cT$) show
\begin{align*} 
S_4&:= 
B_{\NC}(\widehat{\Psi}_{\text{\rm M}}, 
\widehat{\Psi}_{\text{\rm M}}, \widehat{\Phi}_{\text{\rm M}})- B_{\NC}(\Psi_{{\text{\rm M}}},\Psi_{\text{\rm M}}, 
\widehat{\Phi}_{\text{\rm M}})  
=B_{\NC}(\widehat{\Psi}_{\text{\rm M}}-\Psi_{\text{\rm M}}, \widehat{\Psi}_{\text{\rm M}} + \Psi_{\text{\rm M}}, \widehat{\Phi}_{\text{\rm M}}) \nonumber \\
&=  B_{\NC}(\widehat{\Psi}_{\text{\rm M}}+\Psi_{\text{\rm M}}- 2 \Psi,
\widehat{\Psi}_{\text{\rm M}}-\Psi_{\text{\rm M}}, \widehat{\Phi}_{\text{\rm M}})
+ 2 B_{\NC} (\Psi,\widehat{\Psi}_{\text{\rm M}}-\Psi_{\text{\rm M}}, \widehat{\Phi}_{\text{\rm M}}) \nonumber \\
&\lesssim  
\trinl \widehat{\Psi}_{\text{\rm M}}-\Psi_{\text{\rm M}} \trinr_{\NC}  \bigl(
(  \trinl \Psi-\widehat{\Psi}_{\text{\rm M}} \trinr_{\NC} + \trinl \Psi-{\Psi}_{\text{\rm M}} \trinr_{\NC} ) 
\trinl {\widehat{\Phi}}_{\text{\rm M}} \trinr_{\NC}  
+ \| \Psi\|_{{H}^{2+\gamma}(\Omega)} \:
 | \widehat{\Phi}_{\text{\rm M}}   |_{1,2,\pw}  \bigr).  
\end{align*}
{The a~priori error estimate  from Theorem~\ref{aprioriestimate} in terms of the maximal mesh-size $h_{\max}$ yields $\trinl \Psi - \widehat{\Psi}_{\text{\rm M}} \trinr_{\NC}$, $\trinl \Psi -{\Psi}_{\text{\rm M}} \trinr_{\NC} \le C(\gamma, f, \Psi) h^\gamma_{\max}$.} Lemma~\ref{Morley_Interpolation}.a implies  
$ \trinl {\widehat{\Phi}}_{\text{\rm M}} \trinr_{\NC} \le 
 \trinl \Psi-\widehat{\Psi}_{\text{\rm M}} \trinr_{\NC} $. The 
 triangle inequality, Lemma~\ref{Morley_Interpolation}.a, b, 
and Theorem~\ref{aprioriestimate}.c (with respect to $\hT$ rather than $\cT$)
result in
\begin{align*}
| \widehat{\Phi}_{\text{\rm M}}  |_{1,2,\pw}  
&= | \widehat{I}_{\text{\rm M}}\Psi-\widehat{\Psi}_{\text{\rm M}} |_{1,2,\pw} 
\le | {\boldmath \Psi}-\widehat{I}_{\text{\rm M}}\Psi |_{1,2,\pw} 
+ | \Psi-\widehat{\Psi}_{\text{\rm M}} |_{1,2,\pw} \nonumber \\
 &\lesssim   h_{\max}  \trinl {\Psi} - \widehat{I}_{\text{\rm M}}\Psi\trinr_{\pw} 
 + h_{\max}^{\gamma}  \big( \trinl \Psi-\widehat{\Psi}_{\text{\rm M}} \trinr_{\pw} 
 + {\rm osc}_m(f,\hT) \big).
 \end{align*}
%
{This, $1/2 < \gamma \le 1$, and $\trinl \Psi- \widehat{I}_{\text{\rm M}} \Psi \trinr_{\pw} \le 
\trinl \Psi- \widehat{\Psi}_{\M} \trinr_{\pw}$ from Lemma \ref{Morley_Interpolation}.a prove 
\[ S_4 \lesssim  h_{\rm max}^{\gamma}  \big( \trinl \Psi-\widehat{\Psi}_{\text{\rm M}} \trinr^2_{\pw} 
 + {\rm osc}^2_m(f,\hT) \big)^{1/2} \delta(\cT,\hT). \]
 The combination of the established inequalities for $S_3+S_4$ conclude the proof. 
 } 
\end{proof}
   
\subsection{Proof of quasiorthogonality (A4)}\label{subsectProofofquasiorthogonality(A4)}
The proof of ${\bf (A4)}$ departs with a  perturbed form ${\bf (A4)}_{\varepsilon}$ and then employs 
general arguments from the axioms of adaptivity to deduce ${\bf (A4)}$. 
Throughout this section,  let $\cT_k$, $\Psi_k$,  and $\eta_k$ denote the output of AMFEM and abbreviate
$\delta_{k,k+1}:=\delta(\cT_{k},\cT_{k+1})$ for all $k\in\bN_0$ and
\\[1mm]
 \noindent ${\bf (A4)}_{\varepsilon}$ Quasiorthogonality with $\varepsilon >0$.  
 There exists $  0 < \Lambda_{4(\varepsilon)} < \infty $ such that
 \begin{align}\label{quasi}
  \sum_{k=\ell}^{\ell+m} \delta^2_{k,k+1} \le \Lambda_{4(\varepsilon)}  \eta_{\ell}^2 
  + \varepsilon\sum_{k=\ell}^{\ell+m} \eta_k^2
  \quad\text{holds for all }\ell, m \in {\mathbb N}_0.
\end{align}
  
\begin{thm}[Quasiorthogonality]\label{a4e}  For any  {$\varepsilon >0$}  
there exist positive $\delta\le\delta_3$ and $\Lambda_{4(\varepsilon)}$ 
such that  $\cT_0 \in \bT(\delta)$ implies  \eqref{quasi}.
\end{thm}

{
\begin{proof} 
Recall $\Lambda_3$ from the established inequality ${\bf (A3)}$ and set $\widehat{\Lambda}_3:=1+ \Lambda_3$.  Given any positive $\varepsilon$ we may and will assume  without loss of generality that
\begin{equation}\label{eqccnewlastversionforvarepsilon}
0<\varepsilon\le \min \left\{1,\varepsilon_0,  2^{5}  C_{\rm qo}^{1/2} {\widehat \Lambda}_3 /(2+\widehat{\Lambda}_3), 
2^3 \widehat{\Lambda}_3 \right\}.
\end{equation}
Select the maximal  positive  $\delta$ with $\delta\le\min\{\delta_0, \delta_1\}$ and 
\begin{equation}\label{eqccnewlastversionfordelta}
\max\left\{  2^{19/3}  ( C_{\rm qo} \widehat{\Lambda}_3 M\delta)^{2/3}, 
2^6 C_{\rm qo} \widehat{\Lambda}_3  \delta^{2\gamma}, 
3\times 2^9  C_{\rm qo} \widehat{\Lambda}_3 M^2 \delta^2\right\}\le \varepsilon
\end{equation}
 and suppose  $\cT_0\in\bT(\delta)$ so that the maximal mesh-size in any triangulation $\T_k$  is bounded by $h_0\le\delta$. 
Throughout the proof, abbreviate
$e_{k}:=\trinl \Psi- \Psi_{k}\trinr_{\NC} $, $\widehat{e}_k:= \big( e_k^2 + {\rm osc}_0^2(f,\T_k) \big)^{1/2}$, $\delta_{k,k+1}:=\trinl \Psi_{k +1} - \Psi_{k}\trinr_{\NC}$,   and $\mu_k:=\mu(\cT_k)$  with 
 $\mu_k^2(K):=|K|^2(\|f + [u_k,v_k] \|^2_{2,K} + \| [u_k,u_k] \|^2_{2,K}) $ for all $K\in\cT_k$, $k\in\bN_0$.    
Elementary algebra and Lemma~\ref{qo1}  (with $\cT_k$ and   $\cT_{k+1}$ replacing $\cT$ and $\hT$) result in 
\begin{equation}\label{eqnewfinalccproofa4epsiloin1}
 \delta_{k,k+1}^2 + e_{k+1}^2 - e_{k}^2 
 =2  A_{\NC} (\Psi-\Psi_{k+1}, \Psi_{k}- \Psi_{k+1})  \le 2  C_{\rm qo} ^{1/2} \widehat{e}_{k+1}
  \big(  \mu_k(  \cT_k\setminus \cT_{k+1})+  h_0 ^{\gamma} \delta_{k,k+1} \big).
\end{equation}
Two weighted  Young's inequalities for 
$a=2 \widehat{e}_{k+1}$, $b= C_{\rm qo} ^{1/2} \mu_k( \cT_k\setminus \cT_{k+1})$,  $\epsilon=2^{5}  \varepsilon^{-1}  \widehat{\Lambda}_3$ resp.
$a=\delta_{k,k+1}$, $b=  2  C_{\rm qo} ^{1/2}  h_0 ^{\gamma} \widehat{e}_{k+1}$,  $\epsilon=2$ in    \eqref{eq:weighted} 
show that  the right-hand side of \eqref{eqnewfinalccproofa4epsiloin1} is bounded from above by
\[
\delta^2_{k,k+1}/ {4} +{ \varepsilon} 2^{-3}  \widehat{\Lambda}_3^{-1} \widehat{e}^2_{k+1}+ 2^4 \varepsilon^{-1} C_{\rm qo} \widehat{\Lambda}_3 \mu^2_k(  \cT_k\setminus \cT_{k+1})
\]
with  $4 C_{\rm qo}  h_0 ^{2\gamma}\le 2^{-4}   \varepsilon \widehat{\Lambda}_3^{-1}$ from \eqref{eqccnewlastversionfordelta}. 
This leads in  \eqref{eqnewfinalccproofa4epsiloin1} to
\begin{align} \label{quasi2}
\frac{3}4 \delta_{k,k+1}^2 + 
e_{k+1}^2 - e_{k}^2 \le 2^{-3}  {\varepsilon} \widehat{\Lambda}_3^{-1}  \widehat{e}_{k+1}^2 +2^4 {\varepsilon}^{-1} C_{\rm qo} \widehat{\Lambda}_3 \mu_k^2(  \cT_k\setminus \cT_{k+1}).
\end{align}
Lemma~\ref{claim1} (with $\cT_k$ and $\cT_{k+1}$ replacing $\cT$ and $\hT$)  reads
\begin{eqnarray*} 
\mu_k^2(  \cT_k\setminus \cT_{k+1}) +4(\mu^2_{k+1}-\mu^2_k)  &\le &
8\,  M  h_{0} \delta_{k,k+1}(\mu_{k+1}+\mu_k) +  24 \, M^2 h^2_{0}\delta^2_{k,k+1}\\
&\le&   2^{-8} \varepsilon^2  C_{\rm qo}^{-1} \widehat{\Lambda}_3^{-1}  ( \mu_{k+1}^2+\mu_k^2) +  (24+ 2^{13} \varepsilon ^{-2}C_{\rm qo} \widehat{\Lambda}_3) M^2 h^2_{0}\delta^2_{k,k+1}
\end{eqnarray*} 
with a   weighted  Young's inequality for ${a =  8( \mu_{k+1}+\mu_k)}$, $b= M h_{0}\delta_{k,k+1}$, 
$\epsilon= 2^{14}  { \varepsilon^{-2}} C_{\rm qo} \widehat{\Lambda}_3 $  in \eqref{eq:weighted} and a  Cauchy inequality in the last step. 
In the substitution of this estimate in  \eqref{quasi2}, the last term  with $\delta^2_{k,k+1}$ is pre-multiplied by 
$ 2^4 \varepsilon^{-1} C_{\rm qo} \widehat{\Lambda}_3$ and then reads 
\[
 \varepsilon^{-1} C_{\rm qo} \widehat{\Lambda}_3 \left( 
 2^{17}  \varepsilon ^{-2}C_{\rm qo} \widehat{\Lambda}_3+ 3\times 2^{7} 
 \right) M^2 h^2_{0}\delta^2_{k,k+1}\le \delta^2_{k,k+1}/2
\]
from the first and last estimate in \eqref{eqccnewlastversionfordelta}. Therefore and after a multiplication by $4$,  the substitution proves
\begin{align*} 
\delta_{k,k+1}^2
+  4 (e_{k+1}^2 -  e_{k}^2)    \le     \varepsilon    \widehat{e}_{k+1}^2/(2 \widehat{\Lambda}_3) +
 2^{8}  {\varepsilon}^{-1}  C_{\rm qo} \widehat{\Lambda}_3 (\mu^2_k - \mu^2_{k+1})+
 \varepsilon (\mu_{k+1}^2 + \mu_{k}^2)/4.
\end{align*}
The  (partly telescoping) sum over all $k=\ell,\ell+1,\dots,\ell+m$ of this estimate leads to the term
$e_{\ell+m+1}^2 - e_{\ell}^2 $ on the left-hand and  to $\mu_{\ell+m+1}^2 - \mu_{\ell}^2 $ on the
right-hand side in 
\begin{align*} 
& \sum_{k=\ell}^{\ell+m} \delta_{k,k+1}^2 + 4  e_{\ell+m+1}^2-  \varepsilon  \widehat{e}_{\ell+m+1}^2 /(2  \widehat{\Lambda}_3)  
    \le  4 e_\ell^2+    \varepsilon  /( 2 \widehat{\Lambda}_3)\,    \sum_{k=\ell+1}^{\ell+m} \widehat{e}_{k}^2  \notag \\
& \quad \quad     +
 2^{8}{\varepsilon}^{-1}   C_{\rm qo} \widehat{\Lambda}_3 (\mu^2_\ell - \mu^2_{\ell+m+1})+
 \varepsilon /4\,  \sum_{k=\ell}^{\ell+m}   (\mu_{k+1}^2 + \mu_{k}^2).
\end{align*}
Since $\widehat{e}_{\ell+m+1}^2  ={e}_{\ell+m+1}^2 + {\rm osc}_0^2(f, \T_{\ell+m+1}) $, the term $4(1- \varepsilon 2^{-3} \widehat{\Lambda}_3^{-1}) {e}_{\ell+m+1}^2 \ge 0 $
arises  in the lower bound and  is non-negative because of  \eqref{eqccnewlastversionforvarepsilon}.
Recall $[\widehat{u}_k, \widehat{u}_k], [\widehat{u}_k, \widehat{v}_k]  \in P_0(\T_k)$ in the definition of $\mu_k$ to verify ${\rm osc}_0(f, \T_k) \le \mu_k \le \eta_k$. In particular for $k=\ell+m+1$, this and the two contributions of $\mu_{\ell+m+1}$ in the last displayed estimate lead to 
\[
2^{-1}   \varepsilon  \widehat{\Lambda}_3^{-1} {\rm osc}_0^2(f, \T_{\ell+m+1}) + 
\left( \varepsilon/4 - 2^{8}{\varepsilon}^{-1}   C_{\rm qo} \widehat{\Lambda}_3 \right) \mu^2_{\ell+m+1} 
\le \left(
\varepsilon \big( 1+2\widehat{\Lambda}_3^{-1} \big)/4-   2^8 \varepsilon^{-1} C_{\rm qo  } \widehat{\Lambda}_3 \right) \mu^2_{\ell+m+1} \le 0 
\]
from \eqref{eqccnewlastversionforvarepsilon} in the last step. Therefore, the second-last displayed estimate leads to
\[
\sum_{k=\ell}^{\ell+m} \delta_{k,k+1}^2 \le 4 e_\ell^2+   \left( 2^8 \varepsilon^{-1} C_{\rm qo  } \widehat{\Lambda}_3 +\varepsilon/4 \right) \mu^2_{\ell} 
+ \varepsilon/2 \sum_{k=\ell+1}^{\ell+m} \left( \widehat{e}_{k}^2/\widehat{\Lambda}_3 + \mu_k^2 \right).
\]
Recall from the above and Corollary~\ref{correliability} that $\widehat{e}_{k}^2 = e_k^2 + 
{\rm osc}_0^2(f, \T_k) \le \Lambda_3 \eta_k^2 + \mu_k^2 \le \widehat{\Lambda}_3 \eta_k^2$ and  $\mu_k \le \eta_k$ to 
deduce that the last displayed sum is at most $\varepsilon \sum_{k=\ell+1}^{\ell+m} \eta_k^2$.
This concludes the proof of  ${\bf (A4)}_{\varepsilon}$ with 
$ \Lambda_{4(\varepsilon)}:= 4\Lambda_3+2^{8}  \varepsilon^{-1}   C_{\rm qo} \widehat{\Lambda}_3$. 
\end{proof}
}

The refinement rules in AMFEM,  {\bf (A1)}-{\bf (A2)},   
and ${\bf (A4)}_{\varepsilon}$ for small $\varepsilon$ imply  {\bf (A4)}.
   
\begin{cor}[Quasiorthogonality]\label{a4}
Given any  $0<\theta <\theta_0:= 1 /(1+\Lstab^2\LdRel)$, there exists 
a positive $\delta_4\le \delta_3 $ such that  $\cT_0 \in \bT(\delta_4)$ implies  {\bf (A4)}.
\end{cor}

\begin{proof}
Given any $ \theta <\theta_0$ in AMFEM,  {\bf (A1)}-{\bf (A2)} and 
\cite[Thm. 4.1]{cc_hella_18} lead to positive parameters
\(\varrho_{12}<1\) and   \(\Lambda_{12}\) in (A12) undisplayed in this paper. Any choice of   
$\varepsilon<  (1-\varrho_{12})/\Lambda_{12}$
leads in Theorem~\ref{a4e} to some  $\delta_4>0$ so that $\cT_0\in\bT(\delta_4)$ implies 
${\bf (A4)}_{\varepsilon}$.  This and \cite[Thm. 3.1]{cc_hella_18}  imply  {\bf (A4)}. 
\end{proof}

\subsection{Proof of Theorem~\ref{thmOptimalratesAMFEM}}%
\label{ProofofTheoremthmOptimalratesAMFEM}
The assertion follows from  {\bf  (A1)}-{\bf  (A4)} verified for 
$\cT_0,\cT,\hT \in \TT{\delta_4}$  with constants $\Lambda_1,\dots,\Lambda_4$,
which depend on $\cT_{\rm init}$ and on the regular solution $\Psi$. 
Hence Theorem~\ref{thmOptimalrates} applies.  

A closer inspection of the proof of Theorem~\ref{thmOptimalrates} through the axioms of adaptivity in \cite{cc14,cc_hella_18} shows that (given the constants $\Lambda_1,\dots\Lambda_4$) 
only one parameter  $C_{\rm BDD}$  
in the Binev-Dahmen-DeVore theorem  \cite{BinevDahmenDevore04} 
on adaptive mesh-refinement  depends on $\cT_0$
 \cite{BinevDahmenDevore04,Stevenson08}. In fact, 
 $C_{\text{BDV}}$  exclusively depends on shape-regularity  defined in 
\cite[Eq (4.1)]{Stevenson08}. Since uniform mesh-refinement generates  $\cT_0$ in AMFEM, 
 $C_{\text{BDV}}$ exclusively depends on  $\cT_{\rm init}$. 

The abstract analysis in \cite{cc14,cc_hella_18} is rather explicit in the constants and implies
that the equivalence constants in  \eqref{eq:optim} are independent of $\delta$ and $\cT_0$ and exclusively depend on   $\cT_{\rm init}$, $\Lambda_1,\dots\Lambda_4$, and (in a mild way) on $s>0$.
This concludes the proof of Theorem~\ref{thmOptimalratesAMFEM}. \qed

\bibliographystyle{amsplain} 

\section*{Acknowledgements}
The research of the first author has been supported by the Deutsche Forschungsgemeinschaft in the Priority Program 1748 under the project "foundation and application of generalized mixed FEM towards nonlinear problems
in solid mechanics" (CA 151/22-2).  The second author thanks the hospitality of the Humboldt-Universit\"at zu Berlin in August 2017 when this work was initiated. The finalization of this paper has been supported by   DST SERB MATRICS grant MTR/2017/000199 and SPARC project 
(id 235) entitled {\it the mathematics and computation of plates}.

\bibliography{refs_new}
\end{document}